\documentclass[11pt,letterpaper]{article}
\usepackage{amsmath,amssymb,amsthm}
\usepackage{MnSymbol} 
\usepackage{bm, bbm} %
\usepackage{graphicx,color}
\usepackage[hyphens]{url}
\usepackage{dsfont}
\usepackage{booktabs, makecell, adjustbox}
\usepackage[normalem]{ulem}
\usepackage{mathtools}
\usepackage[hidelinks,colorlinks=true,linkcolor=blue,citecolor=blue]{hyperref}
\usepackage[nameinlink,capitalize]{cleveref}
\usepackage{multirow}
\usepackage{algorithm}
\usepackage[noend]{algpseudocode}
\usepackage{tikz} %
\usepackage{pifont}
\usepackage{physics}
\usepackage{enumitem}

\usepackage{enumitem}

\usepackage{soul} %
\usepackage[square,sort,numbers]{natbib}
\usepackage[sort,nocompress,noadjust]{cite}

\usepackage{subcaption}

\usepackage{tabularx}
\newcolumntype{L}{>{\raggedright\arraybackslash}X}

\numberwithin{equation}{section}
\newtheorem{theorem}{Theorem}[section]
\newtheorem{cor}[theorem]{Corollary}
\newtheorem{lemma}[theorem]{Lemma}
\newtheorem{remark}[theorem]{Remark}

\newtheorem{defin}[theorem]{Definition}

\newcommand{\cH}{\mathcal{H}}

\newcommand{\cJ}{\mathcal{J}}

\newcommand{\cM}{\mathcal{M}}

\newcommand{\cP}{\mathcal{P}}
\newcommand{\cQ}{\mathcal{Q}}
\newcommand{\cR}{\mathcal{R}}

\newcommand{\cT}{\mathcal{T}}

\newcommand{\R}{\mathbb{R}}
\newcommand{\C}{\mathbb{C}}

\newcommand*{\E}{\mathbb{E}}

\newcommand{\plusminus}{\raisebox{.2ex}{$\scriptstyle\pm$}}

\DeclareMathOperator*{\argmax}{argmax}

\renewcommand{\epsilon}{\varepsilon}

\newcommand{\imgpad}[1]{%
  \begin{minipage}[c]{.84\linewidth}
  \vspace{3pt}\centering
  \includegraphics[width=\linewidth]{#1}%
  \vspace{3pt}
  \end{minipage}%
}

\crefname{section}{\mbox{\S\!\!}}{\mbox{\S\!\!}}

\newcommand{\cMsym}{\cM_{\mathrm{inv}}}

\usepackage[margin=1in]{geometry}

\makeatletter
\def\blfootnote{\gdef\@thefnmark{}\@footnotetext}
\makeatother

\begin{document}

\title{Min-Max Optimization Is Strictly Easier Than \\ Variational Inequalities}

	 \author{
	 	Henry Shugart
	 	\\	UPenn \\	\texttt{hshugart@upenn.edu}
	 	\and
	 	Jason M. Altschuler
	 	\\	UPenn \\	\texttt{alts@upenn.edu}
	 }
     \date{}
\maketitle

\begin{abstract}
    Classically, a mainstream approach for solving a convex-concave min-max problem is to instead solve the variational inequality problem arising from its
    first-order optimality conditions. Is it possible to solve min-max problems faster by bypassing this reduction? This paper initiates this investigation. We show that the answer is yes in the textbook setting of unconstrained quadratic objectives: the optimal convergence rate for first-order algorithms is strictly better for min-max problems than for the corresponding variational inequalities. The key reason that min-max algorithms can be faster is that they can exploit the asymmetry of the min and max variables---a property that is lost in the reduction to variational inequalities. 
    Central to our analyses are sharp characterizations of optimal convergence rates
    in terms of extremal polynomials which we compute using 
    Green's functions and conformal mappings. 
\end{abstract}

\section{Introduction}\label{sec:intro}

This paper shows a fundamental gap between two well-studied classes of problems. The first is \emph{min-max problems} with convex-concave objectives $f$: find a saddle-point solution $z^* = (x^*,y^*)$ to
\begin{align}\label{eq:minmax}
    \min_x \max_y \, f(x,y)\,.
\end{align}
The second is \emph{variational inequality (VI) problems} with monotone operators $F$: find $z^*$ satisfying
\begin{align}\label{eq:variational}
    \langle F(z^*), \, z-z^* \rangle \geq 0\,, \quad \forall z\,.
\end{align}
Classically, these two problems are intimately connected because the former problem~\eqref{eq:minmax} can be cast as an instance of the latter problem~\eqref{eq:variational} by considering first-order optimality conditions, concatenating the variables $z= (x,y)$, and defining $F(z) = (\nabla_x f(z), -\nabla_y f(z))$ which is guaranteed to be a monotone operator whenever $f$ is convex-concave~\citep{rockafellar1970convex}.

\par Today, this classical connection is central to much of modern algorithm design for min-max problems: simply appeal to standard VI algorithms. This reduction is popular for good reasons: it enables leveraging powerful existing algorithms, it is typically quite effective in both theory and practice, and it is flexible to different problem settings. 
See for example the textbooks \citep{rockafellar1970convex,ryu2022large}. 

\par The ubiquity of this reduction necessitates a fundamental (and remarkably unstudied) question: is it possible to solve min-max problems faster by bypassing this reduction? In other words, does solving the more general problem~\eqref{eq:variational} inherently lead to worse algorithmic guarantees than solving the more specific problem~\eqref{eq:minmax}?

\subsection{Contribution}\label{ssec:cont}

This paper initiates this investigation. We show that the answer is yes in the classical setting of unconstrained quadratic objectives. This uncovers a fundamental gap between the algorithmic complexity of convex-concave min-max problems~\eqref{eq:minmax} and the corresponding VI problems~\eqref{eq:variational}.

\par Specifically, we prove that the optimal convergence rate obtained by first-order algorithms is strictly better for the former than for the latter. We characterize this gap for unconstrained, smooth, and (possibly strongly) convex-concave quadratic $f$ and their corresponding monotone operators $F$. In these settings, we can express the optimal convergence rate in terms of an extremal polynomial problem of the form $\min_{p} \max_{\lambda \in S} |p(\lambda)|$ where $p$ ranges over polynomials whose degree is bounded in terms of the number of iterations that the algorithm is run, and $S$ is a ``spectral range'' (i.e., the set of all possible eigenvalues for an associated linear operator). Importantly, $S$ is an interval for the min-max problem but is a half-disc in $\C$ for the VI problem. In particular, modulo rotation (which is irrelevant for the extremal polynomial problem), the spectral range $S$ in the min-max problem is a strict subset of the spectral range $S$ in the VI problem. This makes the resulting value $\min_{p} \max_{\lambda \in S} |p(\lambda)|$ smaller---and therefore the convergence rate faster---for min-max problems. This gap is precisely quantified by a certain measure of the relative size of the spectral ranges $S$, namely via the ratio of (certain quantities of) the Green's functions for the sets $S$. Combining these ideas, we establish that the optimal convergence rate is faster for min-max problems than VI problems by a factor of $3\sqrt{3}/4 \approx 1.3$ for the strongly-convex-strongly-concave setting and $3\sqrt{3}/2 \approx 2.6$ for the convex-concave setting. Note that in order to prove a separation, we establish a lower bound on the convergence of symmetric algorithms which is slower (by the aforementioned factors) than an upper bound we establish for the convergence of asymmetric algorithms. All our upper and lower bounds are order-optimal. See~\cref{tab:SCSC,tab:CC} for a summary.

\par This modest but fundamental gap uncovers a missed algorithmic opportunity for solving min-max problems.
In particular, our result shows that in order to obtain optimal convergence rates, one must directly design algorithms for min-max problems rather than rely on the classical reduction to VI. This is true even if the min-max problem has identical\footnote{An orthogonal line of work has developed fast asymmetric algorithms for min-max settings in which the optimization problems for $x$ and $y$ have asymmetric structural assumptions, such as differing smoothness or strong convexity parameters~\citep{Lin_Jin_Jordan_2020, Chae_Kim_Kim_2023, wang2020improved, kovalev2022first, heusel2017gans}. The thesis of this paper is complementary: even if the structural assumptions are symmetric in $x$ and $y$, asymmetric algorithms enable faster convergence.} structural assumptions in $x$ and $y$.
A key distinction in such ``direct'' algorithms is that they exploit the intrinsic \emph{asymmetry} of the $x$ and $y$ variables arising in min-max optimization (as opposed to the VI approach which concatenates the variables $z=(x,y)$ at the outset and then treats them symmetrically). 
This asymmetry is a key aspect of a few min-max algorithms, such as gradient-descent-ascent with slingshot stepsizes~\citep{shugart25} or alternating stepsizes~\citep{Zhang_Wang_Lessard_Grosse_2022,Lee_Cho_Yun_2024,Lu_Singh_Chen_Chen_Hong_2019}.
Indeed, a primary motivation of this paper is that the convergence rate of slingshot stepsizes for quadratic min-max problems was better than all existing algorithms for the corresponding VI problems, including even algorithms that use momentum, extragradients, optimism, etc~\citep{shugart25}. The results of this paper show that this gap is fundamental: no symmetric algorithm can converge as fast.

\paragraph*{Outlook.} This paper focuses on demonstrating this phenomenon in its most foundational form. The algorithmic opportunity we uncover opens the door to several new directions for future work, such as showcasing larger gaps in more general settings (e.g., general convex-concave $f$ that are not necessarily quadratic),
exploiting these gaps algorithmically (e.g., as done with our slingshot stepsizes in~\citep{shugart25} for the quadratic setting), and investigating how this gap changes for different algorithm classes (e.g., algorithms that use higher-order information). We believe that this new line of inquiry will lead to a finer-grained complexity of these fundamental problems as well as lead to better algorithms that exploit the missed algorithmic opportunity we uncover. 

\vspace{0.2in}

\begin{table}[H]
\centering
\setlength{\tabcolsep}{6pt}
\renewcommand{\arraystretch}{1.4} %
\begin{tabular}{!{\vrule width 1.2pt}c|>{\centering\arraybackslash}m{0.33\linewidth}|>{\centering\arraybackslash}m{0.33\linewidth}!{\vrule width 1.2pt}}
\Xhline{1.2pt}
\rule{0pt}{1.1em}
 & \textbf{Convex-Concave Min-Max} & \textbf{Monotone VI} \\
\Xhline{1.2pt}
\rule{0pt}{1.1em} Convergence rate &
$\frac{\|\nabla f(z_T)\|}{\|z_0 - z^*\|} \le \frac{L}{T}$ &
$\frac{\|\nabla f(z_T)\|}{\|z_0 - z^*\|} \geq \frac{3\sqrt{3}}{2} \, \frac{L}{T}$ \\[3pt] \hline
\rule{0pt}{1.1em} Spectral range &
\adjustbox{valign=c}{\imgpad{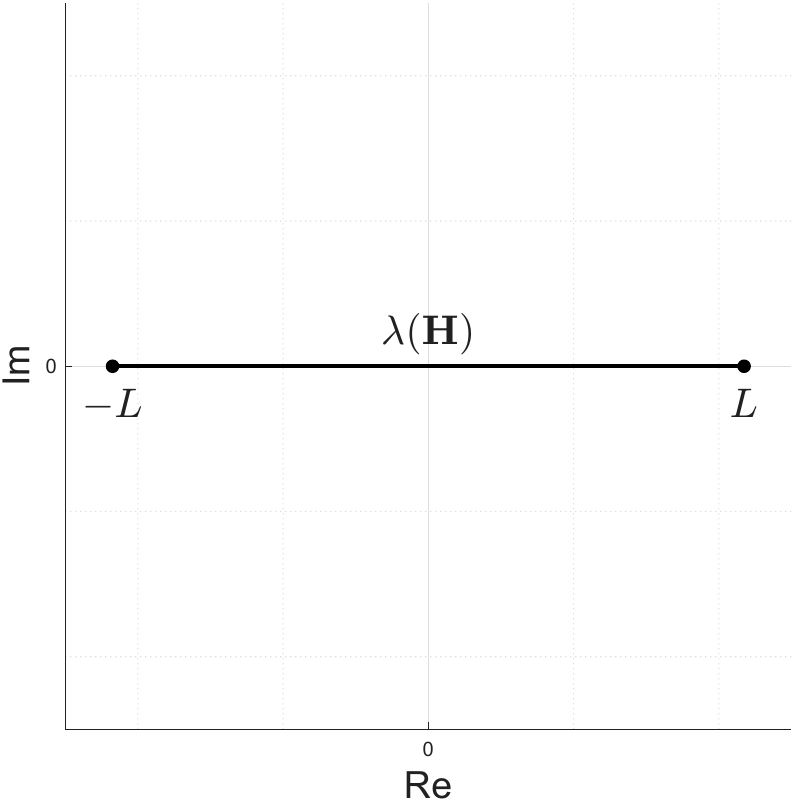}} &
\adjustbox{valign=c}{\imgpad{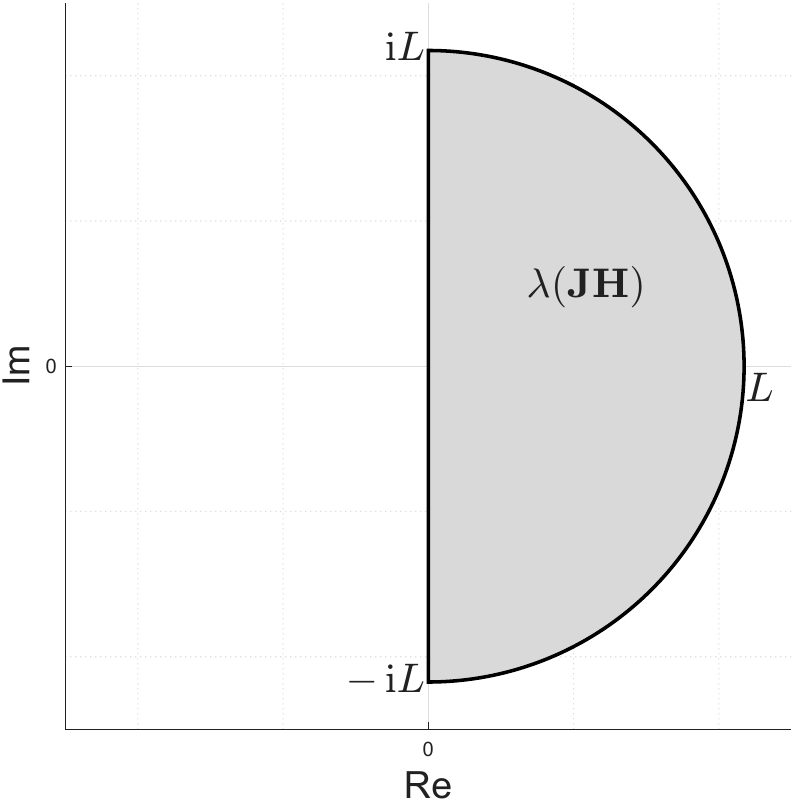}} \\[2pt] \hline
\rule{0pt}{1.1em} Green's function &
$\frac{\partial}{\partial \mathbf{n}} g(0)=1$ &
$\frac{1}{2} \frac{\partial}{\partial \mathbf{n}} g(0)=\frac{2}{3\sqrt{3}}$ \\[2pt]
\Xhline{1.2pt}
\end{tabular}
\caption{\footnotesize \textbf{Top:} We establish a fundamental gap between the fastest possible convergence rate for convex-concave quadratic min-max optimization (\textbf{left}) and the corresponding monotone variational inequalities (\textbf{right}). Here $T$ denotes the number of iterations and $L$ denotes the smoothness.
For simplicity we omit lower-order terms $o(1/T)$; see~\cref{thm:CC-ub,thm:CCLB} for full details. \textbf{Middle:} the underlying geometric reason for this algorithmic gap is that the relevant spectral shape is a strictly smaller subset of $\C$ for min-max problems than for VI problems. \textbf{Bottom:} This gap is precisely quantified by (the derivative of) Green's function for these spectral shapes. The additional factor of $1/2$ appears because the spectral range has positive Lebesgue measure (see~\cref{lem:prescribedroots}).}
\label{tab:CC}
\end{table}

\begin{table}[H]
\centering
\setlength{\tabcolsep}{6pt}
\renewcommand{\arraystretch}{1.4} %
\begin{tabular}{!{\vrule width 1.2pt}c|>{\centering\arraybackslash}m{0.33\linewidth}|>{\centering\arraybackslash}m{0.33\linewidth}!{\vrule width 1.2pt}}
\Xhline{1.2pt}
\rule{0pt}{1.1em} %
 & \textbf{SCSC Min-Max} & \textbf{Strongly-Monotone VI} \\
\Xhline{1.2pt}
\rule{0pt}{1.1em} Convergence rate &
$\frac{\|z_T-z^*\|}{\|z_0 - z^*\|} \le \exp\!\left(-\frac{T}{\kappa}\right)$
&
$\frac{\|z_T-z^*\|}{\|z_0 - z^*\|} \geq \exp\!\left(-\frac{4}{3\sqrt{3}}\frac{T}{\kappa}\right)$ \\[3pt] \hline
\rule{0pt}{1.1em}Spectral range &
\adjustbox{valign=c}{\imgpad{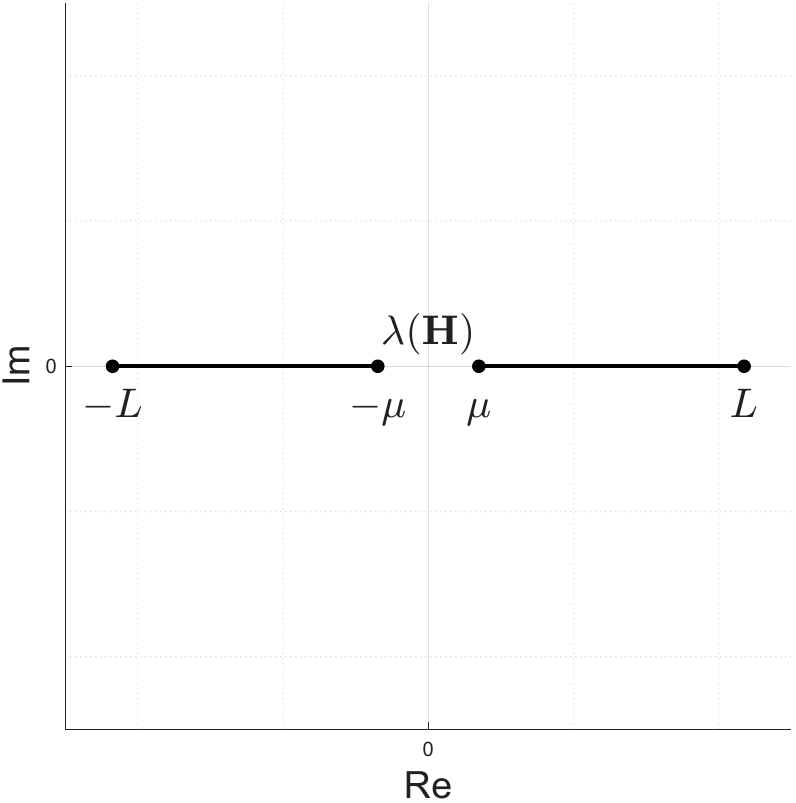}} &
\adjustbox{valign=c}{\imgpad{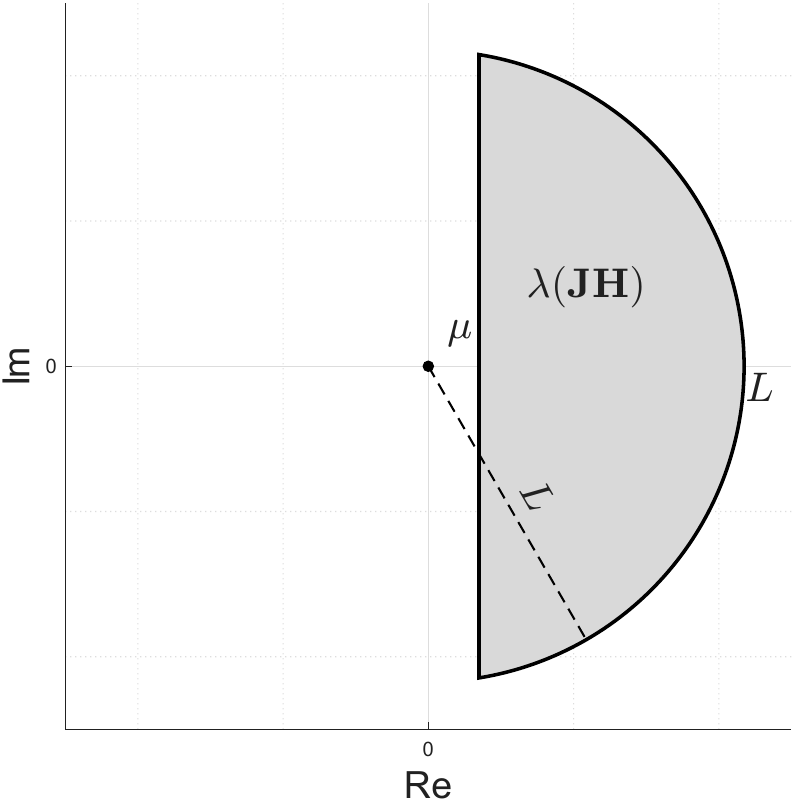}} \\ \hline
\rule{0pt}{1.1em}Green's function &
$g(0)\approx \frac{1}{\kappa}$ &
$g(0)\approx \frac{4}{3\sqrt{3}} \, \frac{1}{\kappa}$ \\[2pt]
\Xhline{1.2pt}
\end{tabular}
\caption{\footnotesize Analog of~\cref{tab:CC} for \textbf{strongly-convex-strongly-concave settings}. Here $\mu$ denotes the strong convexity, $L$ denotes the smoothness, and their ratio $\kappa = L/\mu$ denotes the condition number. For simplicity we omit lower-order terms $o_{T,\kappa}(1)$; see~\cref{thm:SCSCLB,thm:SCSC-ub} for full details.  
}
\label{tab:SCSC}
\end{table}

\section{Preliminaries}\label{sec:prelims}

\subsection{Problem setup}\label{ssec:setup}
\paragraph{Quadratic min-max problems.} We focus on unconstrained min-max problems of the form 
\begin{equation}\label{eq:mm_quad}
\min_{x\in \R^{d_x}} \max_{y\in \R^{d_y}}\ f(x,y)\,  \text{ where } f(x,y) = \frac{1}{2}
\begin{bmatrix} x-x^*\\ y-y^*\end{bmatrix}^{\top}
\underbrace{\begin{bmatrix} \bm A & \bm B\\ \bm B^{\top} & -\bm C\end{bmatrix}}_{\bm H}
\begin{bmatrix} x-x^*\\ y-y^*\end{bmatrix}.
\end{equation}
For convex-concave $f$, stationary points coincide with saddle points; thus all the solutions $z = (x,y)$ of~\eqref{eq:mm_quad} satisfy $z-z^* \in \text{Null}(\bm H)$, where $z^* = (x^*,y^*)$. Indeed this is the criteria for $\nabla f(z) = \bm H (z-z^*)$ to vanish. For simplicity we write problems in the form~\eqref{eq:mm_quad} which is homogeneous around a stationary point $z^*$. This is without loss of generality since, by translating, this captures quadratic objectives $\frac{1}{2}z^\top \bm H z + l^\top z$ with arbitrary linear terms $l$, provided they admit at least one stationary point $\nabla f(z) = 0$. Throughout $\bm A$ and $\bm C$ are assumed symmetric without loss of generality, since otherwise one can replace them with their symmetrizations $(\bm A + \bm{A}^\top)/2$ and $(\bm C + \bm C^\top)/2$.

\paragraph{Reduction to a variational inequality.} The mainstream approach for solving min-max problems is to rewrite the first-order optimality conditions as a variational inequality. The variational inequality associated with~\eqref{eq:mm_quad} is:
\begin{equation} \label{eq:VI_quad}
\text{Find}\  z\in\R^{d_z}\ \text{such that}\  \left\langle F(z),\, z'- z \right\rangle \;\ge 0\,,\; \forall z'\in\R^{d_z}\,, \text{ where } F(z) = \underbrace{\begin{bmatrix} \bm A & \bm B\\ -\bm B^\top & \bm C \end{bmatrix}}_{\bm J \bm H}( z-z^*)\,.
\end{equation}
Above $\bm J = \text{diag}(\bm I, -\bm I)$ and $F = (\nabla_x f, -\nabla_y f)$. 
We write~\eqref{eq:VI_quad} since this is the standard general way to define variational inequalities, although of course in the unconstrained setting,~\eqref{eq:VI_quad} simplifies to finding $z \in \R^{d_z}$ such that $F(z) = 0$. Such a solution corresponds to a stationary point of $f$ and thus a solution of the min-max problem~\eqref{eq:mm_quad}.

\paragraph{Convexity-concavity and monotonicity.} Convergence rates for solving such problems depend on the structure of the objective $f$ and corresponding operator $F$. We focus on the classic setting of (strongly) convex-concave $f$, which corresponds to (strongly) monotone operators $F$. Below we recall these definitions and the correspondences.

\begin{defin}[(Strongly) convex-concave functions]
For $\mu \geq 0$, a function $f(x,y)$ is $\mu$-strongly-convex-strongly-concave ($\mu$-SCSC for short) if $f(\cdot,y)$ is $\mu$-strongly-convex for every $y$, and $f(x,\cdot)$ is $\mu$-strongly-concave for every $x$.\footnote{Recall that a function $g$ is said to be $\mu$-strongly-convex if $g(x) - \tfrac{\mu}{2} \|x\|^2$ is convex. For twice-differentiable $g$, this condition is equivalent to $\nabla^2 g \succeq \mu \bm I$. A function $g$ is said to be $\mu$-strongly-concave if $-g$ is $\mu$-strongly-convex.} If $f$ satisfies this for $\mu=0$, $f$ is convex-concave.
\end{defin}

For quadratic min-max problems~\eqref{eq:mm_quad}, the condition that $f$ be $\mu$-SCSC is equivalent to the condition $\bm A, \bm C \succeq \mu\bm I$. For variational inequalities, the analogous property is (strong) monotonicity.

\begin{defin}[(Strongly) monotone operators]
For $\mu \geq 0$, an operator $F:\R^d\rightarrow \R^d$ is $\mu$-strongly-monotone if $\langle F(z)-F(z'), z-z'  \rangle \geq  \mu \|z-z'\|^2$ for all $z,z'$. If $F$ satisfies this for $\mu=0$, $F$ is monotone.
\end{defin}
In particular, (strong) convexity-concavity of $f$ implies (strong) monotonicity of the saddle operator $F(z) =(\nabla_x f(z), -\nabla_y f(z))$ in the associated variational inequality.
\begin{lemma}[{Convexity-concavity implies monotonicity \citep[Theorem 1]{Rockafellar_1970}}]
Let $\mu \geq 0$. If $f(x,y)$ is $\mu$-strongly-convex-strongly-concave, then $F=(\nabla_x f, -\nabla_y f)$ is $\mu$-strongly-monotone. 
\end{lemma}
\begin{proof}
The cited textbook proves this for $\mu=0$. The proof extends to $\mu > 0$ in a straightforward way: $g(x,y) = f(x,y)-\frac{\mu}{2}\|x\|^2+\frac{\mu}{2}\|y\|^2$ is convex-concave, hence $G = (\nabla_x g, -\nabla_y g)$ is monotone (this is the case $\mu = 0$), hence $\langle G(z) - G(z'), z - z' \rangle \geq 0$ for any $z,z'$. Plugging in $G = F - \mu z$ and re-arranging establishes $\langle F(z) - F(z'), z - z' \rangle \geq \mu \|z - z'\|^2$, hence $F$ is $\mu$-strongly monotone. 
\end{proof}

Aside from monotonicity, throughout we make the standard assumption that $f$ is $L$-smooth, i.e., $\nabla f$ is $L$-Lipschitz. This is equivalent to the saddle-operator $F = (\nabla_x f, -\nabla_y f)$ being $L$-Lipschitz. In the quadratic settings~\eqref{eq:mm_quad} and~\eqref{eq:VI_quad}, this equivalently simplifies to the assumption $\|\bm H\|\leq L$.

\subsection{Spectral range}\label{ssec:spectra}

The spectra of $\bm H$ and $\bm J\bm H$ play a central role in the convergence rate of first-order algorithms.
We start by defining shorthand $\cH_{\mu}$ and $\cJ_{\mu}$ for the sets of possible matrices $\bm H$ and $\bm J\bm H$, respectively, associated with quadratic min-max problems~\eqref{eq:mm_quad} that are $L$-smooth and $\mu$-strongly-convex-concave:
\begin{align}\label{eq:def-matrix-classes}
\mathcal{H}_{\mu} &= \Big\{\bm H : \bm H = \begin{bmatrix} \bm A & \bm B\\ \bm B^\top & -\bm C \end{bmatrix},\, \|\bm H\|\leq L,\, \bm A \succeq \mu \bm I,\, \bm C \succeq \mu \bm I \Big\} \quad \text{ and } \quad 
\mathcal{J}_{\mu} = \Big\{\bm J \bm H : \bm H \in \mathcal{H}_{\mu} \Big\}.
\end{align} 
To avoid notational overhead, we suppress the dependence on $L$ when writing $\cH_{\mu}$ and $\cJ_{\mu}$. We emphasize the dependence on $\mu$ since separating the cases $\mu > 0$ and $\mu = 0$ lets us develop in parallel the strongly and non-strongly convex-concave settings.

\begin{defin}[Spectral range]
The spectral range of a set of matrices $\mathcal{M}$ is
\begin{equation*}
\sigma(\mathcal{M}) = \bigcup_{\bm M \in \mathcal{M}} \sigma(\bm M)\,.
\end{equation*}
where $\sigma(\bm{M})$ denotes the spectrum of a matrix $\bm{M}$.
\end{defin}

The spectral ranges of $\cH_{\mu}$ and $\cJ_{\mu}$ are explicit in terms of $\mu$ and $L$. For $\cJ_{\mu}$ a proof can be found, e.g., in~\citep{Azizian_Scieur_Mitliagkas_Lacoste-Julien_Gidel_2020}. For $\cH_{\mu}$ we are not aware of a reference and therefore provide a short proof here. 

\begin{lemma}[Spectral range of $\mathcal{J}_{\mu}$]\label{lem:spectrum-JH}
For any $0\leq \mu \leq L <\infty$, the spectral range of $\mathcal{J}_{\mu}$ is
$$\sigma({\mathcal{J}_\mu}) = \{z: |z|\leq L, \,\mathrm{Re}(z)\geq \mu\}.$$
\end{lemma}

\begin{lemma}[Spectral range of $\mathcal{H}_{\mu}$]\label{lem:spectrum-H}
For any $0\leq \mu \leq L <\infty$, the spectral range of $\mathcal{H}_{\mu}$ is 
$$\sigma(\mathcal{H}_\mu) = [-L, -\mu] \cup [\mu, L].$$
\end{lemma}
\begin{proof}
The direction ``$\supseteq$'' is clear by considering $\bm A = \bm C = |\lambda|$ and $\bm B = 0$ for any $\lambda \in [\mu,L]$. We prove the other direction ``$\subseteq$'' by combining three observations. First, $\sigma({\mathcal{H}_\mu})$ is real since the matrices $\bm H \in \cH_{\mu}$ are symmetric and thus have real eigenvalues. Second, $\sigma(\cH_{\mu}) \subseteq [-L,L]$ since the spectral radius (i.e., maximum magnitude eigenvalue) of a matrix is upper bounded by the spectral norm, and $\| \bm H \| \leq L$ for all $\bm H \in \cH_{\mu}$. Third, $\sigma(\cH_{\mu}) \subseteq (-\infty,-\mu] \cup [\mu, \infty)$. To show this, it suffices to argue that $\bm H - r \bm I$ is invertible for any $\bm H \in \cH_{\mu}$ and $r \in (-\mu,\mu)$. Observe that the diagonal blocks $\bm A - r \bm I \succeq (\mu - r) \bm I \succ 0$ and $-(\bm C + r\bm I) \preceq -(\mu + r) \bm I \prec 0$ are both invertible. Thus the Schur complement $(\bm A - r \bm I) + \bm B (\bm C + r \bm I)^{-1} \bm B^{\top} \succ 0 $ is also invertible, and hence so is the full matrix $\bm H -r \bm I$, as desired.
\end{proof}

See~\cref{tab:SCSC,tab:CC} for an illustration of these spectral ranges in the cases $\mu > 0$ and $\mu = 0$, respectively. Geometrically, when $\mu > 0$, the spectral range $\sigma({\mathcal{H}_\mu})$ is the union of two real intervals that are symmetric around $0$, whereas $\sigma({\mathcal{J}_\mu})$ is the intersection of the complex disc of radius $L$ with the half-plane $\{z \in \C : \Re(z) \geq \mu\}$. Both sets simplify when $\mu = 0$: then $\sigma(\cH_{\mu})$ becomes a single interval $[-L,L]$ and $\sigma(\cJ_{\mu})$ becomes the half disc of radius $L$ with positive real part.

\subsection{First-order algorithms}\label{ssec:algs}

\paragraph*{First-order algorithms.} We focus on the standard algorithmic model of first-order oracle access to $f$, i.e., black-box queries of the form $f(x,y)$ and $\nabla f(x,y)$. We remark that our proposed algorithms use gradients $\nabla f(x,y)$ but not function evaluations $f(x,y)$; nevertheless we include function evaluations in the definition of the oracle since this is the standard setup. In~\cref{sec:LB} we show that the inclusion of this information does not affect the optimal convergence rates.

\paragraph*{Krylov-subspace algorithms.} We focus on the standard setting of Krylov-subspace algorithms, i.e., algorithms that produce their next iterate within the span of the observed gradients (formal definition below). This is a reasonable assumption since deviating from the span of observed gradients amounts to making an uninformed guess. 
It is classically known from other areas of optimization (see e.g.,~\citep{nemirovsky1991optimality,nemirovsky1992information,nesterov-survey}) that the Krylov-subspace assumption simplifies arguments, isolates the key conceptual ideas, and can be relaxed at the expense of more technical arguments.

\paragraph*{Adaptive algorithms.} All of our results hold regardless of adaptivity, i.e., whether the linear-span coefficients for producing iterates depend on observed information\footnote{
	Formally, an algorithm is said to be non-adaptive if the coefficients $c_{t,k}$ in~\eqref{eq:sum_def_sym} (or analogously $a_{t,k},b_{t,k}$ in~\eqref{eq:sum_def_asym} for asymmetric algorithms) depend on $t$, $k$, $\mu$, and $L$, but \emph{not} on any information from prior iterates. Note that throughout we assume for simplicity that $\mu$ and $L$ are known.  
}. In fact our algorithmic upper bounds are achieved without adaptivity. For clarity of exposition, we first prove our lower bounds for non-adaptive algorithms in~\cref{sec:LB} since these arguments are simpler; then in~\cref{sec:adaptive} we explain the more technical argument for establishing the same results for adaptive algorithms.

\paragraph*{Symmetric vs asymmetric algorithms.} This paper highlights the importance of a complementary axis for distinguishing min-max algorithms: whether updates are symmetric in the minimization variable $x$ and maximization variable $y$. In words, symmetric algorithms treat $x,y$ identically in all updates. Such algorithms are the standard for variational inequality problems since, even from the outset of the problem formulation, variational inequalities do not distinguish between the blocks of $z = (x,y)$; c.f.,~\eqref{eq:mm_quad} versus~\eqref{eq:VI_quad}. 
In contrast, asymmetric algorithms can update $x$ and $y$ differently, which enables exploiting the inherent asymmetry in the definition of min-max problems. Symmetric algorithms can be defined for both min-max problems~\eqref{eq:mm_quad} and variational inequality problems~\eqref{eq:VI_quad}, whereas asymmetric algorithms are only possible to implement for the former.

\begin{defin}[Symmetric algorithms]\label{def:symmetricalgs}
A symmetric Krylov-subspace algorithm is an iterative algorithm that produces a sequence $\{z_t\}$ satisfying, for all $t\ge 1$,
$$
z_t \in z_0 + \mathrm{span}\left\{F(z_0),F(z_1),\,\dots,F(z_{t-1})\right\}.
$$
In other words, there exist coefficients $c_{t,k}$ such that 
\begin{equation}\label{eq:sum_def_sym}
	z_t = z_0 + \sum_{k=0}^{t-1} c_{t,k} F(z_k)\,.
\end{equation}
\end{defin}

\begin{defin}[Asymmetric algorithms]\label{def:asymmetric} 
An asymmetric Krylov-subspace algorithm is an iterative algorithm that produces a sequence $\{z_t = (x_t,y_t)\}$ satisfying, for all $t\ge 1$,
\begin{equation*}
\begin{aligned}
x_t &\in x_0 + \mathrm{span}\left\{ \nabla_x f(x_0,y_0), \dots,\nabla_x f(x_{t-1},y_{t-1}) \right\},\\
y_t &\in y_0 + \mathrm{span}\left\{ \nabla_y f(x_0,y_0), \dots,\nabla_y f(x_{t-1},y_{t-1}) \right\}.
\end{aligned}
\end{equation*}
In other words, there exist coefficients $a_{t,k}$ and $b_{t,k}$ such that 
\begin{equation}\label{eq:sum_def_asym}
	x_t = x_0 + \sum_{k=0}^{t-1} a_{t,k} \nabla_x f(x_k,y_k) \quad \text{ and } \quad y_t = y_0 + \sum_{k=0}^{t-1} b_{t,k} \nabla_y f(x_k,y_k).
\end{equation}
\end{defin}

The main result of this paper is that asymmetric algorithms lead to faster convergence rates than are possible using symmetric algorithms. Our starting point is the observation that the iterates of symmetric and asymmetric algorithms can be associated with polynomials in the matrices $\bm J \bm H$ and $\bm H$, respectively. As we detail formally below, this correspondence is 1-to-1 for the former; whereas for the latter, the class of asymmetric algorithms include these matrix polynomials as an important special case (this inclusion is sufficient for us to develop the claimed faster algorithms). Below and throughout, denote the linear space of bounded-degree polynomials with constant coefficient $1$ by 
\begin{align*}
\mathcal{P}_t = \{p: \mathrm{deg}(p)\leq t,\, p(0)=1\}\,.
\end{align*}

\begin{lemma}[Symmetric algorithms are matrix polynomials of $\bm J \bm H$]\label{lem:sym-algs}
	Consider the min-max problem~\eqref{eq:mm_quad} or the variational inequality problem~\eqref{eq:VI_quad}. 
    For any symmetric Krylov-subspace algorithm, there exists a sequence of polynomials $p_t\in \mathcal{P}_t$ such that for each $t$,
    $$z_t-z^* = p_t(\bm J \bm H) (z_0-z^*).$$
\end{lemma}
\begin{proof}
	We prove by induction on $t$. The base case $t = 0$ is trivial. Supposing true for all $k < t$,  then $		z_t - z^* 
	=
	(z_0 - z^*) + \sum_{k=0}^{t-1} c_{t,k} ( F(z_k) - F(z^*))
	= p_t(\bm J \bm H) (z_0 - z^*)$ where $p_t(\lambda) =1 + \sum_{k=0}^{t-1} c_{t,k} \lambda p_k(\lambda)$.
	Here the first step uses~\cref{def:symmetricalgs} and $F(z^*) = 0$; the second step uses $F(z)= \bm J\bm H (z-z^*)$ and the induction hypothesis $z_k - z^* = p_k(\bm J \bm H)(z_0 - z^*)$.
	Since $p_t \in \cP_t$, the claim is proved.
\end{proof}

\begin{lemma}[Asymmetric algorithms include matrix polynomials of $ \bm H$]\label{lem:asym-algs}
	Consider the quadratic min-max problem~\eqref{eq:mm_quad}.  For any polynomial $p_t\in \mathcal{P}_t$, there exists an asymmetric Krylov-subspace algorithm whose $t$-th iterate satisfies
    \begin{equation*}
        z_t-z^* = p_t(\bm H)(z_0-z^*).
    \end{equation*}
\end{lemma}
\begin{proof}
	Since $p_t \in \cP_t$, it can be written in factorized form $p_t(\lambda) = \prod_{k=0}^{t-1} (1 - \lambda/r_k)$ where $\{r_k\}$ denote its roots. Consider the asymmetric Krylov-subspace algorithm
	\begin{equation}\label{eq:asymmetric-Krylov}
		x_{k+1} = x_{k} - \frac{1}{r_k}\nabla_x f(x_k,y_k) \quad \text{ and } \quad y_{k+1} = y_{k} - \frac{1}{r_k}\nabla_y f(x_k,y_k)
	\end{equation}
	for all $k<t$. By concatenating variables $z = (x,y)$, this can be written succinctly as
	\begin{equation*}
		\begin{aligned}
			z_{k+1}-z^* = z_{k} - z^* - \frac{1}{r_k} \nabla f(z_k)
			=\left(\bm I-\frac{1}{r_k} \bm H \right)(z_{k} - z^*) \,.
		\end{aligned}
	\end{equation*}
	Iterating $t$ times yields the desired identity
	 $z_{t}-z^* = \prod_{k=0}^{t-1}(\bm I-\frac{1}{r_k} \bm H ) (z_0-z^*) = p_t(\bm H)(z_0-z^*)$.
\end{proof}

\begin{remark}[Gradient-descent-ascent with negative/complex stepsizes]
	The asymmetric algorithm we propose in~\eqref{eq:asymmetric-Krylov} is gradient-descent-ascent, but with unconventional stepsizes $\plusminus 1/r_k$ which are negative or complex. Such stepsizes were first explored in our previous work~\citep{shugart25}. By showing that asymmetric algorithms can converge strictly faster than symmetric algorithms, the present paper shows that such stepsizes can lead to faster rates than arbitrary first-order symmetric algorithms---including even algorithms that use momentum, extragradients, optimism, etc. 
    \par Note also that the algorithmic construction of the polynomials in the proof of~\cref{lem:asym-algs} is not unique. Other algorithms could be used, but the key is that any such algorithm must be asymmetric.
\end{remark}

\subsection{Complex analysis and approximation theory preliminaries}\label{ssec:complexanalysis}

Central to our analysis are sharp characterizations of the optimal convergence rates in terms of approximation-theoretic quantities which we compute using Green's functions and conformal mappings. We briefly recall relevant background here for the convenience of the reader. For further details on these topics see for example the excellent textbooks~\citep{stein2010complex,ransford1995potential,trefethen2019approximation}.

\par Below and throughout, we use the following notational shorthands: let $\hat{\C} = \C \cup \{\infty\}$ denote the one-point compactification of $\C$, let $D = \{z : |z| \leq 1\}$ denote the unit disc, and let $\Omega = \{z \in D : \mathrm{Re}(z) \geq 0\}$ denote the unit half disc with positive real part.
Note that $\sigma(\mathcal{J}_0)=\Omega$ is the spectral range for linear monotone operators that are Lipschitz with parameter $L=1$ (see~\cref{lem:spectrum-JH}). Finally, we write $\|p\|_S = \sup_{z \in S} |p(z)|$ to denote the supremum norm of a function $p$ on a set $S$. 

\paragraph*{Conformal mappings.}
Recall that a conformal map is a bijective holomorphic map between open subsets of $\hat{\C}$. The Riemann mapping theorem ensures the existence of conformal maps between any non-empty, simply connected, open sets. We make particular use of the following explicit conformal map of the exterior of the half disc to the exterior of the disc; see~\cref{fig:conformalmap} for a visualization.

\begin{lemma}[Conformal mapping of $\hat{\mathbb{C}}\setminus \Omega$ to $\hat{\mathbb{C}}\setminus D$ \citep{Pommerenke_1961}]\label{lem:conformal}
The function
$$\Phi_\Omega(\lambda) = \frac{\left(1-\left(\frac{\lambda-i}{\lambda+i}\right)^{\frac{2}{3}}\right)-\sqrt{3}i\left(1+\left(\frac{\lambda-i}{\lambda+i}\right)^{\frac{2}{3}}\right)}{2\left(\left(\frac{\lambda-i}{\lambda+i}\right)^{\frac{2}{3}}-1\right)}$$
conformally maps the exterior $\hat{\C} \setminus \Omega$ of the unit half disc to the exterior $\hat{\C} \setminus D$ of the unit disc.
\end{lemma}

\begingroup
\setlength{\intextsep}{-6pt}  
\begin{figure}[H]
    \centering
    \includegraphics[width=0.8\linewidth]{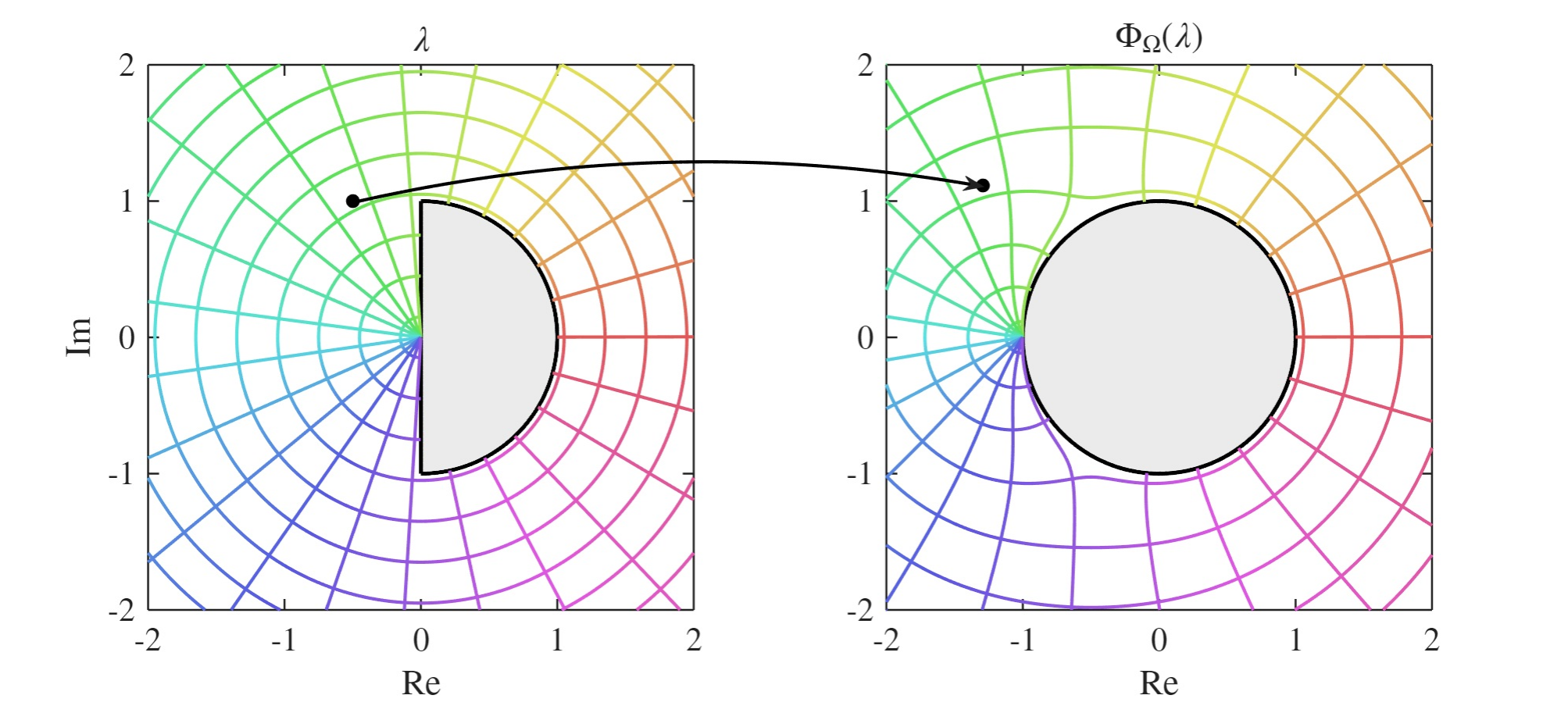}
    \caption{\footnotesize The conformal mapping $\Phi_{\Omega}$ in~\cref{lem:conformal} from the exterior $\hat{\C} \setminus \Omega$ of the half disc to the exterior $\hat{\C} \setminus D$ of the disc. In this plot, a point $\lambda \in \hat{\C} \setminus \Omega$ (\textbf{left}) is mapped to the point $\Phi(\lambda) \in \hat{C} \setminus D$ (\textbf{right}) of the same color. 
    	}
    \label{fig:conformalmap}
\end{figure}
\endgroup

\paragraph*{Green's function.} 
Green's functions arise throughout PDE, complex analysis, potential theory, and more. In this paper we make use of their connections to approximation theory. 
Below and throughout, we consider Green's function with pole at $\infty$ (hence we drop this qualifier as there is no ambiguity) and for sets $S$ that are connected (hence we can use the following formula in terms of conformal mappings).

\begin{defin}[Green's function with pole at $\infty$]\label{def:green}
	Let $S\subset \mathbb{C}$ be a non-empty compact set such that $\hat{ \mathbb{C}}\setminus S$ is simply connected. The Green's function of $S$ is $g_{S}= \log |\Phi_S|$, where $\Phi_S$ is the conformal mapping from $\hat{\mathbb{C}}\setminus S$ to $\hat{\mathbb{C}}\setminus D$ with normalization $\Phi_S(\infty) = \infty$ and $\Phi_S'(\infty) > 0$.
\end{defin}

\cref{fig:greensfunc} illustrates Green's function $g_{\Omega}$ for the unit half disc $\Omega$.  
\begin{figure}[H]
	\centering
	\includegraphics[width=0.4\linewidth]{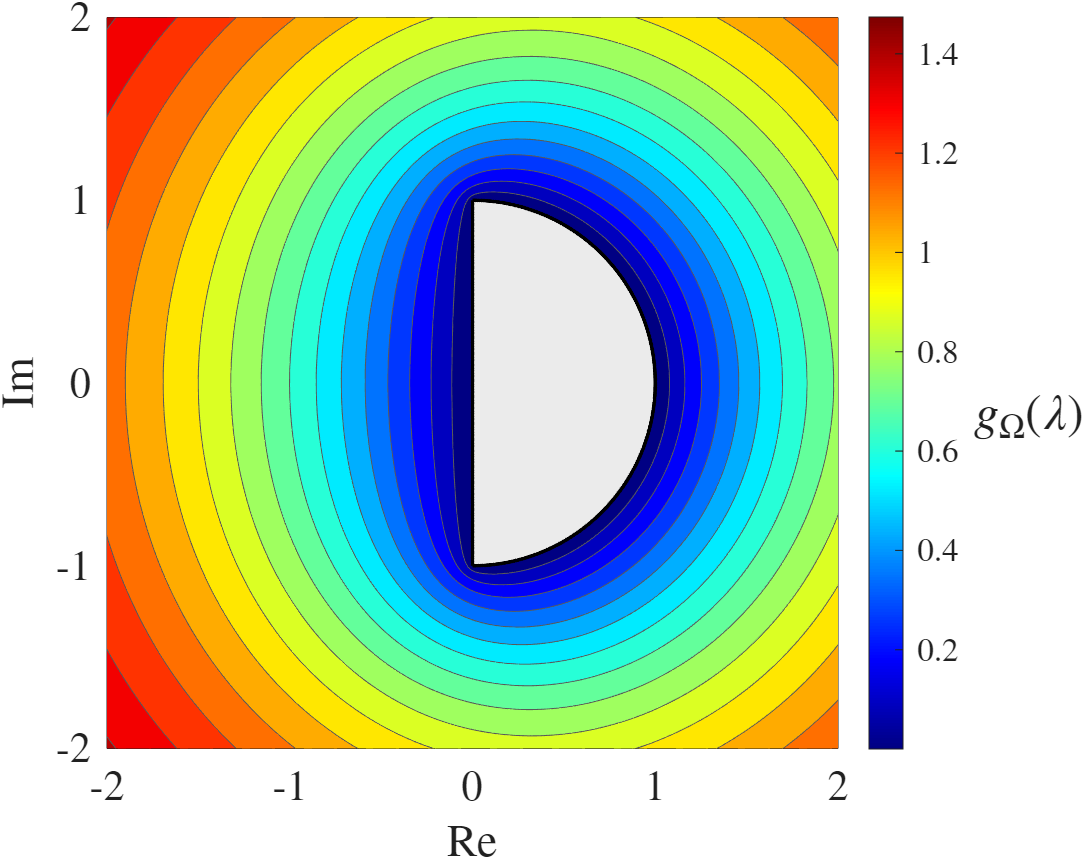}
	\caption{\footnotesize Contour plot of Green's function $g_{\Omega}$ for the unit half disc $\Omega$.}
	\label{fig:greensfunc}
\end{figure}

Green's function arises as a fundamental quantity in our analysis as it captures the maximal growth rate of a polynomial outside of a set $S$, when constrained in sup-norm on $S$. This classical result is due to Bernstein and Walsh \citep{bernstein1912ordre,walsh1926grad}; see for example
Lemma 3.6 of~\citep{saff2010logarithmic} for a short modern exposition, or~\citep[Theorem 1]{Driscoll_Toh_Trefethen_1998} for classical applications to the analysis of matrix iterations. 

\begin{lemma}[Bernstein--Walsh Theorem]
\label{lem:green-poly}
	Let $S\subset \mathbb{C}$ be a non-empty compact set such that $\hat{ \mathbb{C}}\setminus S$ is simply connected. For any $\lambda \in \C \setminus S$ and any polynomial $p$ of degree at most $T$,
	\begin{equation*}
		|p(\lambda)| \leq e^{Tg_S(\lambda)} \|p\|_S.
	\end{equation*}
\end{lemma}

Relatedly, Green's function also captures the maximal growth rate of the \textit{derivatives} of polynomials. Such growth bounds are often called Bernstein-type inequalities (recalled in~\cref{ssec:CCLB} when we make use of them) and depend on Green's function through $\frac{\partial}{\partial \mathbf{n}} g_{S}(\lambda)$ which denotes, for $\lambda \in \partial S$, the derivative of $g_S$ in the direction normal to $S$ (with orientation pointing away from $S$). For ease of recall, we state here that the explicit value of this quantity for $S=\Omega$ and $\lambda=0$.

\begin{cor}\label{cor:green_partial} 
	$\frac{\partial}{\partial \mathbf{n}} g_{\Omega}(0) =\frac{4}{3\sqrt{3}}$.
\end{cor}
\begin{proof}
	By~\cref{lem:conformal} and~\cref{def:green}, 
	$\frac{\partial}{\partial \mathbf{n}} g_{\Omega}(0) =\frac{\partial}{\partial \mathbf{n}} \log |\Phi_{\Omega}(\lambda)| = \frac{|\Phi'_\Omega(0)|}{|\Phi_\Omega(0)|} = \frac{4/(3\sqrt{3})}{1}$.	
\end{proof}

\section{Lower bounds for symmetric algorithms}\label{sec:LB}
In this section we prove lower bounds on the computational complexity of  \textit{symmetric} first-order algorithms for solving (strongly) convex-concave min-max problems. This immediately implies analogous lower bounds for solving variational inequalities for (strongly) monotone operators by the standard reduction in~\cref{ssec:setup}. For simplicity of exposition, here we restrict to non-adaptive algorithms; in~\cref{sec:adaptive} we show how the results extend to adaptive algorithms.

\subsection{Strongly-convex-strongly-concave problems}\label{ssec:SCSCLB}

Our first result is for the strongly-convex-strongly-concave (SCSC) setting. Recall that $\kappa$-conditioned means $\mu$-SCSC and $L$-smooth for condition number $\kappa = L/\mu$. 
This result improves on the previous state of the art lower bound of $\exp(-2 T/\kappa)$ (up to lower order terms in $T,\kappa$) given in \citet[Proposition 2]{ibrahim2020linear}. Our improved bound allows us to show a computational gap between symmetric and asymmetric algorithms for the first time---since the improvement of this lower bound from $\exp(-2 \cdot T/\kappa)$ to roughly $\exp(-(4\sqrt{3}/3) \cdot T/\kappa) \approx \exp(-0.77 \cdot T/\kappa)$ crosses the threshold of the $\exp(-1 \cdot T/\kappa)$ upper bound established for asymmetric algorithms later in~\cref{thm:SCSC-ub}. 

\begin{theorem}[Lower bound for symmetric algorithms on SCSC problems]\label{thm:SCSCLB}
    For any non-adaptive symmetric Krylov-subspace algorithm and any number of iterations $T$, there exists a $\kappa$-conditioned quadratic min-max problem~\eqref{eq:mm_quad} with solution $z^*$ such that the convergence rate is no faster than
    \[
        \frac{\|z_T-z^*\|}{\|z_0 - z^*\|} \geq \Bigg|\Phi_{\Omega}\Bigg(\frac{-1}{\kappa-1}\Bigg)\Bigg|^{-T} =  \exp\Bigg( -\Bigg(\frac{4}{3\sqrt{3}} +o_{\kappa}(1)\Bigg) \frac{T}{\kappa} \Bigg) \,.
    \]
\end{theorem}
\begin{proof}
By~\cref{lem:sym-algs}, the $T$-th iterate $z_T$ of the algorithm can be expressed as
\begin{equation}\label{eq:pf-SCSCLB-poly}
    z_T-z^* = p_T(\bm J \bm H)(z_0-z^*)\,,
\end{equation}
where $p_T \in \cP_T$, i.e., $p_T$ is a polynomial of degree at most $T$ satisfying the normalization constraint $p_T(0)=1$.
In the worst-case over $\kappa$-conditioned $\bm H$ and solutions $z^*$, the convergence rate is no faster than
\begin{align*}
\max_{\bm H,\, z^*} \frac{\|z_T-z^*\|}{\|z_0-z^*\|} = \max_{\bm H,\, z^*} \frac{\left\|p_T(\bm J \bm H)(z_0-z^*) \right\|}{\|z_0 - z^*\|}
= \max_{\bm H}\left\|p_T(\bm J\bm H) \right\| 
\geq \max_{\bm H} \left|\lambda_{\max}\left(p_T(\bm J\bm H)\right) \right|
\geq 
\max_{\lambda\in \sigma(\mathcal{J}_\mu)}|p_T(\lambda)|\,.
\end{align*}
Above, the first step is by definition of $p_T$ in~\eqref{eq:pf-SCSCLB-poly}, the second step is by definition of the operator norm, the third step is because the operator norm of a matrix is bounded below by its spectral radius, and the final step is because the eigenvalues of $p_T(\bm J \bm H)$ are simply given by $p_T$ evaluated at the eigenvalues of $\bm J \bm H$ (see for example Theorem 1.1.6 of~\citep{horn2012matrix}).
\par We conclude that the fastest possible convergence rate for such algorithms is no faster than 
\begin{align}
    \min_{p \in \cP_T} \max_{\lambda \in \sigma (\cJ_{\mu})} |p(\lambda)|\,.
    \label{eq:SCSCLB-poly}
\end{align}
Classical tools from approximation theory let us lower bound such extremal polynomial problems via an associated Green's function (\cref{lem:green-poly}). To invoke this, it is convenient to first replace $\sigma(\mathcal{J}_\mu)$ by a slightly smaller set with an explicit Green's function, namely the half-disc $\Omega_\mu := \{\lambda: |\lambda-\mu|\leq L-\mu, \text{Re}(\lambda)\geq \mu\}$ that has center $\mu$ and radius $L-\mu$. Since $\sigma(\mathcal{J}_\mu) \supset \Omega_{\mu}$ by~\cref{lem:spectrum-JH}, an application of \cref{lem:green-poly} with $S = \Omega_{\mu}$ and $\lambda = 0$ gives
\begin{align*}
    \min_{p\in \cP_T} \max_{\lambda\in \sigma(\mathcal{J}_\mu)}|p(\lambda)|
    \geq \min_{p\in \cP_T} \max_{\lambda\in \Omega_\mu}|p(\lambda)|
    \geq
    \exp( -T g_{\Omega_\mu}(0))\,.
\end{align*}
This quantity involving the Green's function is explicitly computed as:
\begin{align*}
\exp( -T g_{\Omega_{\mu}}(0))= \left|\Phi_{\Omega_{\mu}}\left(0\right)\right|^{-T}
=\left|\Phi_{\Omega}\left(\frac{-1}{\kappa -1}\right)\right|^{-T}
=
\exp\left( -\left(\frac{4}{3\sqrt{3}} +o_{\kappa}(1)\right) \frac{T}{\kappa} \right)
\,.
\end{align*}
Above, the first step is by definition of the Green's function $g_{\Omega_{\mu}} = \log |\Phi_{\Omega_{\mu}}|$ in terms of the conformal map $\Phi_{\Omega_{\mu}}$ from the exterior of $\Omega_{\mu}$ to the exterior of the unit disc (\cref{def:green}), the second step is by recentering and rescaling $\Omega_{\mu}$ to the standard half disc $\Omega$, and the final step is by plugging in the explicit conformal mapping $\Phi_{\Omega}$ (\cref{lem:conformal}) and using the asymptotic expansion $|\Phi_{\Omega}(-\tfrac{1}{\kappa - 1})|^{-1} = 1 - \tfrac{4}{3\sqrt{3} \kappa} + O(\tfrac{1}{\kappa^2}) = \exp(-\tfrac{4}{3\sqrt{3} \kappa} + O(\tfrac{1}{\kappa^2}))$ for $\kappa \to \infty$.
\end{proof}

\subsection{Convex-concave problems}\label{ssec:CCLB}
We now turn to the convex-concave setting, i.e., $\mu=0$. Our main result here is the following lower bound for symmetric algorithms. This sharpens the well-known $O(1/T)$ lower bound~\citep{Yoon_Ryu_2021} sufficiently to establish a separation from the faster rates of asymmetric algorithms shown later. 

\begin{theorem}[Lower bound for symmetric algorithms on convex-concave problems]\label{thm:CCLB}
For any non-adaptive symmetric Krylov-subspace algorithm and any number of iterations $T$, there exists an $L$-smooth quadratic min-max problem~\eqref{eq:mm_quad} with solution $z^*$ such that the convergence rate is no faster than
    $$\|\nabla f(z_T)\|\geq \left(\frac{3\sqrt{3}}{2}+o_T(1) \right)\frac{L\|z_0-z^*\|}{T}.$$
\end{theorem}

Note that convergence is measured via the gradient norm $\|\nabla f(z_T)\|$. Although distance to optimum $\|z_T - z^*\|$ is a meaningful metric in the SCSC setting (c.f.\ \cref{thm:SCSCLB}), it is well-known that such convergence rates are impossible in the convex-concave setting due to pathologically flat objectives. 

A core ingredient in our proof of~\cref{thm:CCLB} is a Bernstein-type inequality, i.e., an inequality bounding the derivative of a polynomial by the largest value that the polynomial takes on a given set. Recall that a Jordan curve is the image of an injective continuous map of a circle, and a Jordan arc is the image of an injective continuous map of a line segment; we will apply this for the Jordan curve being the boundary of the half-disc $\partial\Omega$ (the boundary of the relevant spectral shape for convex-concave problems, see~\cref{lem:spectrum-JH}) and the Jordan arc being a small interval of the imaginary axis around $0$ (the relevant prescribed root, see the proof of~\cref{thm:CCLB} below).

\begin{lemma}[Bernstein inequality for polynomials with prescribed zeros on general domains]\label{lem:prescribedroots}
Let $K\subset \mathbb{C}$ be a compact set bounded by a Jordan curve. Let $\lambda_0$ be a point on the boundary of $K$, and assume the boundary of $K$ is a twice continuously differentiable Jordan arc in a neighborhood of $\lambda_0$. Let $p_T$ be a polynomial of degree at most $T$. Further assume $\lambda_0$ is a root of $p_T$. Then
$$|p_T'(\lambda_0)| \leq \left(1+o_T(1)\right) \cdot \frac{T}{2} \cdot \frac{\partial}{\partial \mathbf{n}}  g_K(\lambda_0) \cdot \|p_T\|_K.$$
\end{lemma}

\cref{lem:prescribedroots} combines two strengthenings of Bernstein's classical inequality \citep{bernstein1912ordre}: strengthened Bernstein inequalities for polynomials with prescribed zeros \citep{Rahman_Mohammad_1967} (note that $\lambda_0$ is a root of $p_T$) and Bernstein inequalities on general domains \citep{Nagy_Totik_2005,Nagy_2005} (Bernstein's original inequality was only for the disc). Theorem 1.3 of~\citep{Nagy_Totik_2005} establishes~\cref{lem:prescribedroots} but with twice as large an upper bound and without the assumption that $\lambda_0$ is a root.~\cref{lem:prescribedroots} follows by the same proof---simply replace the use of Bernstein's standard inequality with the strengthening in~\citep{Rahman_Mohammad_1967} which improves the bound by a factor of $2$ when $\lambda_0$ is a root.\footnote{Details: simply carry through the twofold improvement in the three proof steps. 1) Replace the standard Bernstein inequality with the twofold improvement of~\citep[Corollary 1]{Rahman_Mohammad_1967} in~\citep[page 452]{Nagy_2005}.
2) Carry through this tightened bound on page 455 to improve~\citep[Theorem 1]{Nagy_2005} by a factor of $2$. 3) Carry through this tightened bound in equation (3) on page 193 to improve~\citep[Theorem 1.3]{Nagy_Totik_2005} by a factor of $2$.}
This twofold improvement enables our theory to tightly bound the relevant extremal polynomial~\eqref{eq:CCLB-poly}, described below, as can be verified numerically.

\begin{proof}[Proof of~\cref{thm:CCLB}] We begin by reducing to an extremal polynomial problem, similarly to our analysis of the strongly-convex-strongly-concave setting (see the proof of~\cref{thm:SCSCLB}). As done there, let $p_T$ denote the polynomial corresponding to the algorithm when run for $T$ iterations. Then the worst-case convergence rate is no faster than
\begin{equation*}
\max_{\bm H,\, z^*} \frac{\|\nabla f(z_T)\|}
{\|z_0-z^*\|} = \max_{\bm H,\, z^*} \frac{\left\|\bm J \bm H p_T(\bm J \bm H)(z_0-z^*) \right\|}{\|z_0 - z^*\|}
= \max_{\bm H} \|\bm J \bm H p_T(\bm J \bm H)\|
\geq \max_{\lambda \in \sigma(\cJ_0)} |\lambda p_T(\lambda)|\,.
\end{equation*}
Above, the first step is because $\|\nabla f(z_T)\| = \|\nabla f(z_T) - \nabla f(z^*)\| = \| \bm H (z_T - z^*)\| = \| \bm H p_T(\bm J \bm H) (z_0 - z^*)\| = \| \bm J \bm H p_T(\bm J \bm H) (z_0 - z^*)\|$. Introducing the extra factor of $\bm J$ here does not affect the norm but ensures that the final expression in the above display depends on the matrix $\bm H$ only through $\bm J \bm H$.  
\par We conclude that the fastest possible convergence rate for such algorithms is no faster than
\begin{align}\label{eq:CCLB-poly}
\min_{p\in \cP_T} \max_{\lambda\in \sigma(\mathcal{J}_0)}|\lambda p(\lambda)| = \min_{q\in \mathcal{Q}_{T+1}} \max_{\lambda\in \sigma(\mathcal{J}_0)}|q(\lambda)|
\end{align}
where we define the shorthand $\mathcal{Q}_{T+1}=\{q: \textrm{deg}(q)\leq T+1,\; q(0)=0,\; q'(0)=1\}$. Note that the normalization is different than in~\cref{ssec:SCSCLB}, including constraints on both the function value and first derivative. Since $0$ is a root of $q\in \mathcal{Q}_T$, the Bernstein inequality~\cref{lem:prescribedroots} gives
\begin{align*}
|q'(0)| &\leq \left(1+o_T(1)\right) \cdot \frac{T}{2}\cdot \frac{\partial}{\partial \mathbf{n}} g_{\sigma(\mathcal{J}_0)}(0) \cdot \max_{\lambda\in \sigma(\mathcal{J}_0)}|q(\lambda)|\,.
\end{align*}
\par We now compute the quantity $\frac{\partial}{\partial \mathbf{n}} g_{\sigma(\mathcal{J}_0)}(0)$. Recall from~\cref{lem:spectrum-JH} that $\sigma(\cJ_0)$ is the half disc of radius $L$. This quantity is explicit for the half disc $\Omega$ of radius $1$ by~\cref{cor:green_partial}. To rescale appropriately, first observe that $g_{\sigma(\mathcal{J}_0)}(\lambda)= \log|\Phi_{\sigma(\mathcal{J}_0)}(\lambda)|=\log|\Phi_\Omega(\frac{\lambda}{L} )|= g_{\Omega}(\frac{\lambda}{L})$ by using the definition of Green's function in terms of the  conformal mapping and rescaling the conformal mapping. Hence by the chain rule and then~\cref{cor:green_partial}, 
\begin{align*}
    \frac{\partial}{\partial \mathbf{n}} g_{\sigma(\mathcal{J}_0)}(0)=\frac{1}{L}\frac{\partial}{\partial \mathbf{n}} g_{\Omega}(0)= \frac{4}{3\sqrt{3}L}\,.
\end{align*}
Combining the above three displays with the fact that $|q'(0)| = 1$ is normalized for all $q \in \cQ_T$ yields
\begin{align*}
    \min_{q \in \cQ_T}
    \max_{\lambda\in \sigma(\mathcal{J}_0)}|q(\lambda)|
    \geq 
    \left(\frac{3\sqrt{3}}{2}+o_T(1) \right)\frac{L\|z_0-z^*\|}{T}\,.
\end{align*}
\end{proof}

\section{Upper bounds for asymmetric algorithms}\label{sec:UB}
In this section we present \emph{asymmetric} algorithms that break the convergence rate lower bounds established in~\cref{sec:LB} for \emph{symmetric} algorithms. We do this for both the strongly and non-strongly convex-concave settings. In all settings, we use gradient-descent-ascent (GDA) with the slingshot stepsize schedules proposed in our previous work \citep{shugart25}.

\par We begin with brief background on slingshot stepsizes. Recall that the update of GDA is
\begin{equation*}
	\begin{aligned}
		x_{t+1} &= x_t - \alpha_t \nabla_x f(x_t,y_t) \\
		y_{t+1} &= y_t + \beta_t \nabla_y f(x_t,y_t).
	\end{aligned}
\end{equation*}
Note that GDA is a symmetric algorithm if and only if $\alpha_t=\beta_t$ for all $t$. Following \citep{shugart25}, we use slingshot stepsize schedules of the form
\begin{equation}
	\alpha_{2t} = -\beta_{2t}=-\alpha_{2t+1}=\beta_{2t+1}=h_t\,,  \quad \text{for } t=0,...,T/2-1\,,
\end{equation}
for appropriate choices of the magnitudes $h_t$. 
\begin{figure}
	\centering
	\includegraphics[width=0.65\linewidth]{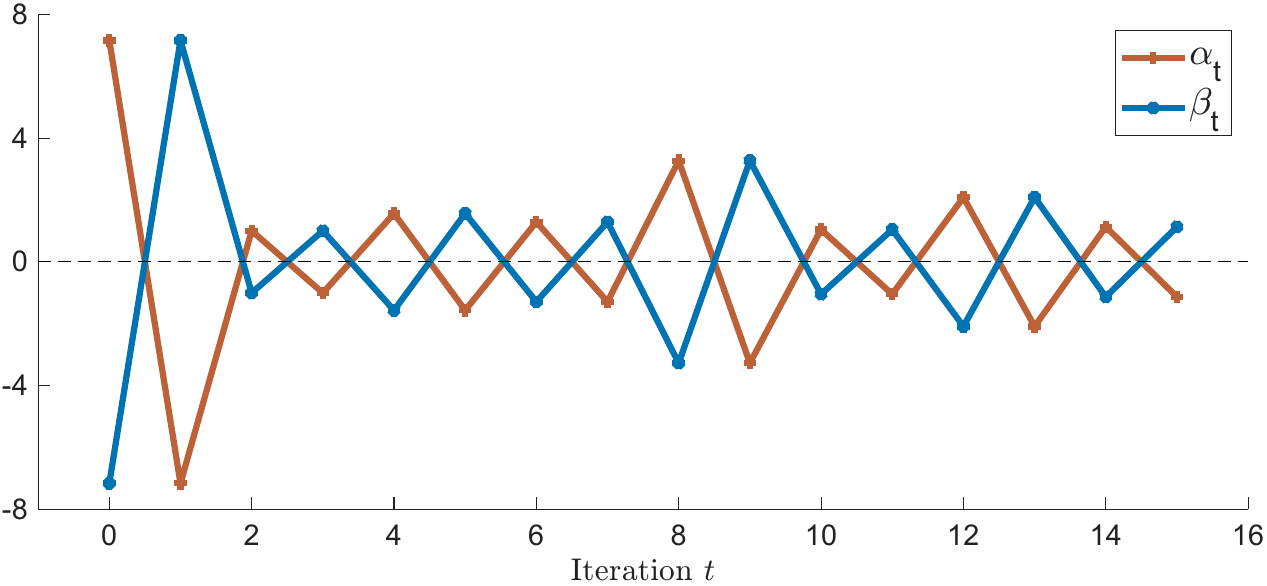}
	\caption{\footnotesize 
        Slingshot stepsize schedule (\cref{def:steps-SCSC}) 
        for 
        $\mu=0.1$, $L=1$, and $T=16$. 
    }
	\label{fig:steps}
\end{figure}
See~\cref{fig:steps} for an illustration. These stepsize schedules are asymmetric, time-varying, and periodically negative. They are implementable for min-max problems, but not for general variational inequality problems due to the asymmetry in $x$ and $y$; see \S1.2 of~\citep{shugart25} for a detailed discussion. By concatenating the variables $z = (x,y)$, this stepsize schedule results in paired updates of the form
\begin{align*}
	z_{2t+1} &= z_{2t} - h_t \nabla f(z_{2t})\,, \\
	z_{2t+2} &= z_{2t+1} + h_t \nabla f(z_{2t+1})\,,
\end{align*}
which for quadratic objectives $f(z) = \tfrac{1}{2} (z-z^*)^T \bm H (z - z^*)$ as in~\eqref{eq:mm_quad}, yields the two-step update
\begin{equation*}
	z_{2t+2}-z^* = \left( \bm I + h_t \bm H\right) \left( \bm I - h_t \bm H\right)(z_{2t}-z^*) = \left(\bm I -h_t^2 \bm H^2\right)(z_{2t}-z^*).
\end{equation*} 
For $T$ even, this results in a $T$-step cumulative update
\begin{align}\label{eq:matrix_poly_SSS}
	z_{T} - z^* = \prod_{t<T/2} (\bm I - h_t^2 \bm H^2) (z_0 - z^*)
\end{align}
which is given by a matrix polynomial of the Hessian $\bm H$ with roots $\pm 1/h_t$. We choose the stepsize magnitudes $\{h_t\}$ explicitly in terms of the roots of certain Chebyshev polynomials; the optimal choice depends on the problem setting---see~\cref{ssec:SCSCUB} and~\cref{ssec:CCUB} below for the strongly and non-strongly convex-concave settings, respectively.

\subsection{Strongly-convex-strongly-concave problems} \label{ssec:SCSCUB}

For this setting, we choose the slingshot stepsize magnitudes $\{h_t\}$ in terms of the roots of the degree-$T/2$ Chebyshev polynomial  $\cT_{T/2}^{[\mu^2,L^2]}$ of the first kind on the interval $[\mu^2, L^2]$. For background on Chebyshev polynomials, see for example the textbooks \citep{mason2002chebyshev,rivlin2020chebyshev}. 

\begin{defin}[Slingshot stepsize schedules for strongly-convex-strongly-concave min-max problems]\label{def:steps-SCSC}
	For any even number of iterations $T=2N$ and any parameters $0 < \mu \leq L < \infty$, the slingshot stepsize schedule for strongly-convex-strongly-concave problems is 
	\begin{align*}
		\alpha_{2t} = -\beta_{2t} = -\alpha_{2t+1} = \beta_{2t+1} = h_t\,, \qquad  t \in \{0,1,...,N-1\}\,,
	\end{align*}
	where $\{h_t\}_{t=0}^{N-1}$ are any permutation of $\{r_t^{-1/2}\}_{t=0}^{N-1}$, where 
	\begin{align*}
		r_t := \frac{L^2+\mu^2}{2} + \frac{L^2-\mu^2}{2} \cos\left( \frac{2t+1}{T}\pi\right)\,, \qquad t \in \{0,1,...,N-1\}\,,
	\end{align*}
	are the $N$ roots of the Chebyshev polynomial $\cT_N^{[\mu^2,L^2]}$.
\end{defin}

Two remarks. First, despite the present setting (strongly-convex-strongly-concave quadratics) being different from the setting in~\citep[\S3.1]{shugart25} (bilinear objectives with non-negative singular values), the stepsizes in~\cref{def:steps-SCSC} exactly coincide with the ones we proposed in that paper. 
This is because in both problem settings, the Hessian $\bm H$ has the same spectral range.
Second, note that the order of the steps in~\cref{def:steps-SCSC} does not matter due to commutativity of the updates, at least in exact arithmetic implementations; see Appendix A of~\citep{shugart25} for fractal-like orderings that improve numerical stability. 

\begin{theorem}[Upper bound for strongly-convex-strongly-concave min-max problems]\label{thm:SCSC-ub}
	Consider any even integer $T$, any dimensions $d_x, d_y$, any initialization $z_0 = (x_0,y_0)\in \R^{d_x}\times \R^{d_y}$, and any $\mu$-strongly-convex-strongly-concave quadratic min-max problem that is $L$-smooth. Using the slingshot stepsize schedule in \cref{def:steps-SCSC}, GDA converges to the unique saddle point $z^*$ at rate
	\begin{equation}
        \frac{\|z_{T} -z^*\|}{\|z_0 - z^*\|} \leq \frac{2(\kappa + 1)^{T/2} (\kappa - 1)^{T/2}}{(\kappa+1)^{T} + (\kappa-1)^{T}}  = \exp\left( -\left(1+ o_{T,\kappa}(1) \right) \frac{T}{\kappa}\right)\,.
	\end{equation}
\end{theorem}
\begin{proof}
	By~\eqref{eq:matrix_poly_SSS} and~\cref{def:steps-SCSC}, we can write the $T$-step update of GDA with the proposed slingshot stepsizes as the following matrix polynomial of $\bm H^2$:
	\begin{equation*}
		z_T-z^* = \prod_{t< T/2} \left(\bm I -h_t^2 \bm H^2\right) (z_{0}-z^*)
		=\prod_{t < T/2} \left(\bm I - \bm H^2/r_t\right) (z_{0}-z^*)
		= \frac{\cT_{T/2}^{[\mu^2,L^2]}(\bm H^2)}{\cT_{T/2}^{[\mu^2,L^2]}(0)}  (z_{0}-z^*)\,.
	\end{equation*}
	The convergence rate then follows:
	\begin{equation*}
		\|z_T-z^*\| \leq \frac{\|\cT_{T/2}^{[\mu^2,L^2]}(\bm H^2)\| }{\cT_{T/2}^{[\mu^2,L^2]}(0)} \,\|z_0-z^*\|\\
		\leq \frac{2(\kappa + 1)^{T/2} (\kappa - 1)^{T/2}}{(\kappa+1)^{T} + (\kappa-1)^{T}}  \,\|z_0-z^*\|.
	\end{equation*}
	Above, the first step is by sub-multiplicativity of the operator norm. The second step uses two classical facts about Chebyshev polynomials (see e.g., Lemma 3.2 of~\citep{shugart25}), namely the closed-form expression for $\cT_{T/2}^{[\mu^2,L^2]}(0) = \tfrac{(\kappa+1)^{T} + (\kappa-1)^{T}}{2(\kappa + 1)^{T/2} (\kappa - 1)^{T/2}}$ and the fact that $\sup_{\lambda \in [\mu^2, L^2]} |\cT_{T/2}^{[\mu^2,L^2]}(\lambda)| = 1$, which is applicable since $\bm H^2$ is diagonalizable with eigenvalues in $[\mu^2,L^2]$ by~\cref{lem:spectrum-H}.     
	\par Finally, a Taylor expansion of the rate gives the desired asymptotics
	$$\frac{2(\kappa + 1)^{T/2} (\kappa - 1)^{T/2}}{(\kappa+1)^{T} + (\kappa-1)^{T}} = \exp\left(- \left( 1+o_{\kappa,T}(1)\right) \frac{T}{\kappa} \right).$$
\end{proof}

\subsection{Convex-concave problems} \label{ssec:CCUB}

For this setting, we can directly invoke the optimal convergence rate proven in our prior work~\citep{shugart25}. For the convenience of the reader, below we recall these optimal stepsizes and the corresponding convergence rate. These appear originally as Definition 3.5 and Theorem 3.7 in~\citep{shugart25}. 

\begin{defin}[Slingshot stepsize schedule for convex-concave min-max problems]\label{def:quadsteps}
	For any even number of iterations $T$ and any $L$-smooth, convex-concave, quadratic min-max problem, the slingshot stepsize schedule is 
	\begin{equation}
		\nonumber
		\alpha_{t} = -\beta_{t} = h_t,\qquad  t \in \{0,1,...,T-1\},
	\end{equation}
	where $\{h_t^{-1}\}_{t=0}^{T-1}$ are any permutation of 
	\begin{align}
		\nonumber
		\rho_t := L \cos \left( \frac{2t+1}{2T+2} \pi \right)\,, \qquad t \in \{0,\dots,T\} \setminus \{T/2\}\,,
	\end{align}
	which are the $T$ non-zero roots of the Chebyshev polynomial $\cT_{T+1}^{[-L,L]}$. Notice that like in~\cref{def:steps-SCSC}, these roots come in positive/negative pairs because 
	\[
	\rho_t = -\rho_{T-t}\,.
	\]
\end{defin}
\begin{theorem}[Upper bound for convex-concave min-max problems]\label{thm:CC-ub}
	Consider any even integer $T$, any dimensions $d_x, d_y$, any initialization $z_0 = (x_0,y_0)\in \R^{d_x}\times \R^{d_y}$, and any quadratic min-max problem that is convex-concave and $L$-smooth. Using the slingshot stepsize schedule in \cref{def:quadsteps}, GDA converges at rate
	\begin{equation}
		\|\nabla f(z_{T})\| \leq \frac{L}{T+1} \|z_0 - z^*\|\,,
	\end{equation}
	where $z^*$ is any saddle point.
\end{theorem}
This convergence rate is exactly optimal---among not just arbitrary stepsize schedules for GDA, but in fact among arbitrary first-order algorithms---as it exactly matches the lower bound in \citep{Yoon_Ryu_2021}.

\section{Adaptive algorithms}\label{sec:adaptive}

In this section we extend the lower bounds for symmetric algorithms in \cref{sec:LB} to allow for the algorithms to be \emph{adaptive}. Conceptually, this amounts to constructing a \emph{single} problem instance for which no (adaptive) algorithm can perform well---in contrast to the arguments in~\cref{sec:LB} which construct a potentially different hard problem for each (non-adaptive) algorithm. 

\par These results extend for both the strongly-convex-strongly-concave and convex-concave settings. For the convenience of the reader, we state both results below in their entirety; the only difference from~\cref{thm:SCSCLB,thm:CCLB}, respectively, is that the results here allow the algorithms to be adaptive. 

\begin{theorem}[Lower bound for adaptive symmetric algorithms on SCSC problems]\label{thm:SCSCLB_adapt}
    There exists a $\kappa$-conditioned quadratic min-max problem~\eqref{eq:mm_quad} such that for any (possibly adaptive) symmetric Krylov-subspace algorithm and any number of iterations $T$, the convergence rate from some initialization point $z_0$ is no faster than
    \[
        \|z_T-z^*\| \geq    \exp\Bigg( -\Bigg(\frac{4}{3\sqrt{3}} +o_{\kappa}(1)\Bigg) \frac{T}{\kappa} \Bigg) \|z_0 - z^*\| \,.
    \]
\end{theorem}

\begin{theorem}[Lower bound for adaptive symmetric algorithms on convex-concave problems]\label{thm:CCLB_adapt}
    There exists an $L$-smooth quadratic min-max problem~\eqref{eq:mm_quad} such that for any (possibly adaptive) symmetric Krylov-subspace algorithm and any number of iterations $T$, the convergence rate from some initialization point $z_0$ is no faster than
    $$\|\nabla f(z_T)\|\geq \left(\frac{3\sqrt{3}}{2}+o_T(1) \right)\frac{L\|z_0-z^*\|}{T}.$$
\end{theorem}

\subsection{Proof of~\cref{thm:SCSCLB_adapt}}\label{ssec:adaptive-sketch}

The proofs of~\cref{thm:SCSCLB_adapt,thm:CCLB_adapt} follow from nearly identical extensions of~\cref{thm:SCSCLB,thm:CCLB}, respectively. Therefore for brevity we only detail the former, i.e., how to prove~\cref{thm:SCSCLB_adapt} from~\cref{thm:SCSCLB}.  For conceptual clarity, we first outline the argument. For shorthand, throughout this section we denote $\sigma(\cJ_{\mu})$ simply by $\sigma$, since here there is no possibility of confusion with $\sigma(\cH_{\mu})$. Recall from~\cref{lem:spectrum-JH}
that $\sigma = \{z: |z|\leq L, \mathrm{Re}(z)\geq \mu\}$.

\paragraph*{Conceptual overview: hard problem instances for adaptive algorithms, via duality of the extremal polynomial problem.} Recall from \cref{sec:LB} that our lower bounds begin with a reformulation of the optimal convergence rate in terms of an extremal polynomial problem $\min_{p \in \cP_T} \max_{\lambda \in \sigma} |p(\lambda)|$, where the degree-$T$ polynomial $p$ encodes the $T$-iteration evolution of the algorithm, and the eigenvalue $\lambda$ of $\bm J \bm H$ encodes a worst-case problem instance for that algorithm. 
See~\eqref{eq:SCSCLB-poly}.
Importantly, as we considered only non-adaptive algorithms there, the polynomial $p_T$ did not depend on the matrix $\bm J \bm H$, so we could select a hard problem $\lambda$ based on the algorithm $p_T$ (hence the order of the $\min$ and $\max$ in $\min_p \max_{\lambda} |p(\lambda)|$).  
\par For adaptive algorithms, however, the polynomial $p_T$ may depend on the matrix $\bm J \bm H$, so we identify a single problem instance which is hard for all algorithms. In terms of the extremal polynomial problem, this requires a dual form where the order of the min and max are swapped. In particular, one may hope to prove a duality statement of the form 
 \begin{equation}\label{eq:maxmin}
 \min_{p\in \mathcal{P}_T} \max_{\lambda \in \sigma} |p(\lambda)|^2 = \max_{\nu \in \mathcal{M}(\sigma) } \min_{p\in \mathcal{P}_T} \E_{\lambda \sim \nu} |p(\lambda)|^2\,.
 \end{equation}
where $\mathcal{M}(\sigma)$ is the set of probability measures supported on $\sigma$. Notice that the maximum over eigenvalues $\lambda
$ is lifted to a maximum over \emph{probability distributions} on eigenvalues $\nu
$; in game-theoretic terminology this corresponds to relaxing pure strategies to mixed strategies. 

\par At a conceptual level, our overall proof strategy amounts to proving a duality statement of the form~\eqref{eq:maxmin}, showing the existence of a solution $\nu$, and then using $\nu$ to construct a hard problem instance for which the optimal convergence rate of any (possibly adaptive) symmetric algorithm is given by $\min_{p\in \mathcal{P}_T} \E_{\lambda \sim \nu} |p(\lambda)|^2$. This style of argument is inspired by Nemirovsky's classical lower bounds for solving symmetric linear systems using Krylov-subspace algorithms~\citep{nemirovsky1991optimality, nemirovsky1992information}, at least at a high level. However, implementing this proof strategy leads to additional technical considerations for min-max problems than in the simpler setting of quadratic minimization studied by Nemirovsky. Below we detail our two key steps to overcome these hurdles.

\paragraph*{Step 1: Nearly-optimal distributions with finite support and conjugation invariance.} The method we use to construct a hard problem from $\nu$ is not amenable to arbitrary probability distributions $\nu$. We therefore show that we can impose two constraints on $\nu$ that on one hand make our construction possible, and on the other hand only affect the final convergence rate by an arbitrarily small amount. The first condition we impose on $\nu$ is finite support; later this will enable us to construct a hard problem instance in finite dimension. The second condition we impose on $\nu$ is invariance under conjugation $\nu(\lambda)= \nu(\bar \lambda)$; later this will enable us to construct a hard problem whose operator $\bm J \bm H$ and solution $z^*$ have all real entries. 
\begin{lemma}[Step 1 in proof of~\cref{thm:SCSCLB_adapt}]\label{lem:adaptive1}
For any number of iterations $T$ and any error $\varepsilon>0$, there exists a probability distribution $\nu \in \cM(\sigma)$ satisfying the following:
\begin{enumerate}
\item[(i)] $\nu$ is finitely supported.
\item[(ii)] $\nu$ is invariant under conjugation.
\item[(iii)] $\left(1-\varepsilon \right)^2\min_{p\in \mathcal{P}_T} \max_{\lambda \in \sigma} |p(\lambda)|^2 \leq\min_{p\in \mathcal{P}_T} 
\E_{\lambda \sim \nu} |p(\lambda)|^2$.
\end{enumerate}
\end{lemma}

We prove this lemma in~\cref{ssec:adaptproof1}; here we focus on how we use it to prove~\cref{thm:SCSCLB_adapt}.

\paragraph{Step 2: Constructing a hard problem instance from $\nu$.} 
The second step of our argument uses $\nu$ to construct a ``hard'' quadratic min-max problem, as formally stated next.

\begin{lemma}[Step 2 in proof of~\cref{thm:SCSCLB_adapt}]\label{lem:adaptive2}
Suppose there exists a probability distribution $\nu \in \cM(\sigma)$ satisfying the following:
\begin{enumerate}
\item[(i)] $\nu$ is finitely supported.
\item[(ii)] $\nu$ is invariant under conjugation.
\item[(iii)] $R \leq\min_{p\in \mathcal{P}_T} 
\E_{\lambda \sim \nu} |p(\lambda)|^2$.
\end{enumerate}
Then there exists a $\kappa$-conditioned quadratic min-max problem~\eqref{eq:mm_quad} such that for any (possibly adaptive) symmetric Krylov-subspace algorithm and any number of iterations $T$, the convergence rate from some initialization point $z_0$ is no faster than
$\|z_T - z^*\| \geq \sqrt{R} \|z_0-z^*\|$.
\end{lemma}
\begin{proof}
    Let $S$ denote the support of $\nu$; this is finite by property (i). Define $S^+ = \{ \lambda \in S : \Im(\lambda) \geq 0\}$. We construct the $\kappa$-conditioned quadratic min-max problem~\eqref{eq:mm_quad}. Define the blocks of the quadratic objective $\bm H$ as $\bm A = \bm C = \mathrm{diag}( \{ \Re(\lambda_j) : \lambda_j \in S^{+}\} )$ and $\bm B = \mathrm{diag}( \{ \Im(\lambda_j) : \lambda_j \in S^{+}\}$. 
Then, modulo a permutation of rows and columns, the matrix
\begin{align*}
    \bm J \bm H = \begin{bmatrix}
        \bm A & \bm B \\
        -\bm {B^\top} & \bm{C}
    \end{bmatrix}\,,
\end{align*}
is block-diagonal with $2 \times 2$ blocks of the form
$$\begin{bmatrix} \text{Re}(\lambda_j) & \text{Im}(\lambda_j) \\ -\text{Im}(\lambda_j) & \text{Re}(\lambda_j) \end{bmatrix}\,.$$
for every $\lambda_j \in S^+$. The eigenvalues of each such block are $\lambda_j$ and $\bar \lambda_j$. (Note that for real $\lambda_j$, this $2 \times 2$ block is diagonal with $\lambda_j$ repeated twice.) Therefore the spectrum of $\bm J \bm H$ coincides with $S$ because of the conjugation invariance in property (ii). 
\par Finally we construct the initialization $z_0$ and solution $z^*$ so that distance to optimality on each eigenspace is $\langle u_j, z_0-z^*\rangle = c_j$ where $c_j = \sqrt{\nu(\lambda_j)/2}$ if $\lambda_j$ is real and $c_j = \sqrt{\nu(\lambda_j)}$ otherwise, where $u_j$ is the eigenvector associated with the eigenvalue $\lambda_j$. (The extra factor of $2$ accounts for the double-counting of real eigenvalues described above.) This is achieved, for example, by the explicit construction $z^* = z_0 - \sum_j c_j ( e_{j} + e_{j + |S^+|} )$.

\par We now prove that this construction witnesses the desired lower bound for the convergence rate of any symmetric Krylov-subspace algorithm. Recall from~\cref{ssec:algs} that the $T$-th iterate $z_T$ of any such algorithm satisfies $z_T - z^* = p_T(\bm J \bm H)(z_0 - z^*)$ for some polynomial $p_T \in \cP_T$. Thus
\begin{align*}
\|z_T-z^*\|^2 &= \|p_T(\bm J\bm H)(z_0-z^*)\|^2 = \sum_{j} \left(|p_T(\lambda_j)| \cdot \langle u_j, z_0 - z^* \rangle \right)^2 =   \E_{\lambda \sim \nu} |p(\lambda)|^2 \geq R\,.
\end{align*}
Above, the first step is by definition of $p_T$, the second step is by block-diagonalizing, the third step is by definition of $\nu$, and the final step is by property (iii). 
\end{proof}

Combining~\cref{lem:adaptive1} and~\cref{lem:adaptive2} immediately implies~\cref{thm:SCSCLB_adapt}. It remains only to prove~\cref{lem:adaptive1}, which we do below.

\subsection{Proof of \cref{lem:adaptive1}}\label{ssec:adaptproof1}

\subsubsection{Helper lemmas}
We begin with three helper lemmas. The first constructs the support of $\nu$, which we denote by $S$. Note that for our purposes, $|S|$ need not be controlled so long as it is finite for every fixed $T$ and $\varepsilon$. 

\begin{lemma}[Finite mesh of $\sigma$]\label{lem:finite}
For every positive integer $T$ and error $\varepsilon > 0$, there exists a finite subset $S\subset \sigma$ satisfying
$$ (1-\varepsilon)\|p\|_{\sigma} \leq \|p\|_{S} $$
for all polynomials $p$ of degree at most $T$.
\end{lemma}

\begin{proof}
Define $\sigma^+ = \{\lambda^+ : \exists \lambda \in \sigma,\, |\lambda^+ -\lambda|\leq 1 \}$ to be the set of points of distance at most 1 from $\sigma$. Since the closure $\text{cl}(\sigma^+\setminus \sigma)$ is compact and $g_{\sigma}$ is continuous on it, $g_{\max} = 
\|g_{\sigma}\|_{\text{cl}(\sigma^+\setminus  \sigma)}
$ is finite.
Let $S$ be any $\delta$-net of $\partial \sigma$ for $\delta = \varepsilon \exp( - T g_{\max})$; that is, let $S$ be a finite subset of $\partial \sigma$ such that for every point in $\partial \sigma$ there exists a point in $S$ within distance $\delta$.

Now fix any polynomial $p$ of degree at most $T$.
By Cauchy's estimate\footnote{For completeness, we recall here Cauchy's estimate: $|f^{(n)}(\lambda)| \leq \frac{n!}{r^n} \sup_{\lambda' :|\lambda'-\lambda|\leq r} |f(\lambda')|$ for any holomorphic $f$, $n\in \mathbb{N}$, $r>0$, and $\lambda\in \mathbb{C}$. A proof can be found, e.g., in~\citep[identity 2.14]{conway2012functions}.} and then an extremal growth bound for polynomials via Green's function (\cref{lem:green-poly})
$$
\|p'\|_{\sigma}
\leq \|p\|_{\sigma^+}
\leq \exp(T g_{\text{max}}) \|p\|_{\sigma}\,.
$$
Now let $\lambda^* \in \argmax_{\lambda \in \sigma} |p(\lambda)|$. By the maximum modulus principle, $\lambda^* \in \partial \sigma$. By definition of $S$ as a $\delta$-net of $\partial \sigma$, there exists $\lambda \in S$ for which $|\lambda - \lambda^*| \leq \delta$. By the above display, it follows that $$\big| |p(\lambda)| - |p(\lambda^*)|\big| \leq \delta \|p'\|_{\sigma} \leq \delta \exp(Tg_{\max}) \|p\|_{\sigma} = \varepsilon \|p\|_{\sigma}\,.$$
Since $|p(\lambda^*)| = \|p\|_{\sigma}$ by definition of $\lambda^*$, it follows that $|p(\lambda)| \geq (1 - \varepsilon)\|p\|_{\sigma}$. This proves the lemma since $\lambda \in S$ is arbitrary. 
\end{proof}

We also make use of the following two elementary helper lemmas about symmetry of polynomials along the real axis.
Below, recall the notation that $\cP_T$ is the linear space of polynomials of degree at most $T$ satisfying $p(0) = 1$, and let $\cR_T$ denote the subspace of $\cP_T$ containing only polynomials with real coefficients.

\begin{lemma}[Helper lemma 1] 
\label{lem:helper-real-max}
Suppose $S \subseteq \mathbb{C}$ is closed under conjugation. Then $$\min_{r \in \cR_T} \|r\|_S = \min_{p \in \cP_T} \|p\|_S\,.$$
\end{lemma}
\begin{proof}
Denote $\tilde{p} = \overline{p(\overline{\lambda})}$ and $r = \tfrac{1}{2} (p + \tilde{p})$. Since $S$ is closed under conjugation, $\|p\|_S = \|\tilde{p}\|_S$. Thus $\|r\|_S = \tfrac{1}{2} \, \|p + \tilde{p}\|_S \leq \max\{\|p\|_S,\, \|\tilde{p}\|_S \} = \|p\|_S$. Finally, to check that $r$ has real coefficients, note that if $p(\lambda) = \sum_{t} c_t \lambda^t$, then $\tilde{p}(\lambda) = \sum_t \bar{c}_t\lambda^t$, hence $r = \sum_t \Re(c_t) \lambda^t$.
\end{proof}

\begin{lemma}[Helper lemma 2]\label{lem:helper-real-avg}
    Suppose $\nu$ is invariant under conjugation. Then
    \begin{align*}
        \min_{r \in \cR_T} \E_{\lambda \sim \nu} |r(\lambda)|^2 
        =   
        \min_{p \in \cP_T} \E_{\lambda \sim \nu} |p(\lambda)|^2 
    \end{align*}
\end{lemma}
\begin{proof}
    The direction ``$\geq$'' is obvious since $\cP_T \supset \cR_T$. For the direction ``$\leq$'', observe that by linearity of expectation, it suffices to show that for any $p \in \cP_T$, there exists $r \in \cR_T$ satisfying 
$        |r(\lambda)|^2 + |r(\bar{\lambda})|^2 \leq 
        |p(\lambda)|^2 + |p(\bar{\lambda})|^2  $
    for all $\lambda \in \C$. To this end, define $r$ as in the proof of~\cref{lem:helper-real-max}; i.e., let $\tilde{p} = \overline{p(\overline{\lambda})}$ and $r = (p + \tilde{p})/2$. Then $r \in \cR_T$ (as shown there) and 
    satisfies the desired inequality since
    \begin{align*}
        |r(\lambda)|^2 + |r(\bar{\lambda})|^2
        = 
        2|r(\lambda)|^2
        =
        \frac{1}{2}\, |p(\lambda) + p(\bar{\lambda})|^2
        \leq
        |p(\lambda)|^2 + |p(\bar{\lambda})|^2\,.
    \end{align*}
    Above, the first step is because $r(\lambda)$ and $r(\bar{\lambda})$ are complex conjugates and thus have the same magnitude; the second step is by definition of $r$; and the final step is by the elementary inequality $|a+b|^2 \leq 2(|a|^2 + |b|^2)$ for any $a,b \in \C$.
\end{proof}

\subsubsection{Combining the helper lemmas}

\begin{proof}[Proof of~\cref{lem:adaptive1}]
    Let $S$ be the finite set in~\cref{lem:finite}. Without loss of generality, suppose $S$ is closed under conjugation (since otherwise we can include all conjugates). For shorthand, let $\cMsym$ denote the subset of probability distributions in $\cM(S)$ that are invariant under conjugation. 
    \par Using in order: the definition of $S$,~\cref{lem:helper-real-max}, linearity of expectation, Sion's minimax theorem, a symmetrization argument (replacing $\nu(\lambda)$ by $(\nu(\lambda) + \nu(\bar{\lambda}))/2$), and then~\cref{lem:helper-real-avg}, we conclude
    \begin{align*}
        (1-\varepsilon)^2\, \min_{p\in \mathcal{P}_T} \max_{\lambda \in \sigma} |p(\lambda)|^2 
         &\leq 
         \min_{p\in \mathcal{P}_T} \max_{\lambda \in S} |p(\lambda)|^2 
         \\ &=  \min_{r \in \mathcal{R}_T} \max_{\lambda \in S} |r(\lambda)|^2
         \\ &=  \min_{r \in \mathcal{R}_T} \max_{\nu \in \cM(S)} \E_{\lambda \sim \nu} |r(\lambda)|^2
         \\ &= \max_{\nu \in \cM(S)}  \min_{r \in \mathcal{R}_T} \E_{\lambda \sim \nu} |r(\lambda)|^2
         \\ &= \max_{\nu \in \cMsym(S)}  \min_{r \in \mathcal{R}_T} \E_{\lambda \sim \nu} |r(\lambda)|^2
         \\ &= \max_{\nu \in \cMsym(S)}  \min_{p \in \mathcal{P}_T} \E_{\lambda \sim \nu} |p(\lambda)|^2\,.
    \end{align*}
    Let $\nu$ be an optimal solution to the final expression; existence is guaranteed by compactness. Then $\nu$ satisfies all three desired properties: the first by finiteness of $S$, the second because $\nu \in \cMsym(S)$, and the third by the above display.
\end{proof}

\paragraph*{Acknowledgements.} We are grateful to Joel Tropp for helpful discussions about the literature. JMA acknowledges funding from a Sloan Research Fellowship and a Seed Grant Award from Apple.

 \small
\addcontentsline{toc}{section}{References}
\bibliographystyle{plainnat}
\bibliography{refs}

@article{shugart25,
  title={Negative Stepsizes Make Gradient-Descent-Ascent Converge},
  author={Shugart, Henry and Altschuler, Jason M},
  journal={Preprint at arXiv:2505.01423},
  year={2025}
}

@article{Pommerenke_1961, title={On metric properties of complex polynomials.}, volume={8}, number={2}, journal={Michigan Mathematical Journal}, publisher={University of Michigan, Department of Mathematics}, author={Pommerenke, Christian}, year={1961}, pages={97–115} }

@article{Nagy_Totik_2005, title={Sharpening of {H}ilbert’s lemniscate theorem}, volume={96}, number={1}, journal={Journal d’Analyse Mathématique}, author={Nagy, Béla and Totik, Vilmos}, year={2005}, pages={191–223}, language={en} }

@article{Nagy_2005, title={Asymptotic {B}ernstein inequality on lemniscates}, volume={301}, ISSN={0022-247X}, number={2}, journal={Journal of Mathematical Analysis and Applications}, author={Nagy, Béla}, year={2005}, pages={449–456} }

@article{Rahman_Mohammad_1967, title={Remarks on {S}chwarz’s lemma}, volume={23}, number={1}, journal={Pacific Journal of Mathematics}, publisher={Mathematical Sciences Publishers}, author={Rahman, Qazi and Mohammad, Q. G.}, year={1967}, pages={139–142} }

@article{Driscoll_Toh_Trefethen_1998, title={From Potential Theory to Matrix Iterations in Six Steps}, volume={40}, number={3}, journal={SIAM Review}, author={Driscoll, Tobin A. and Toh, Kim-Chuan and Trefethen, Lloyd N.}, year={1998}, pages={547–578}, language={en} }

@inproceedings{Yoon_Ryu_2021, title={Accelerated Algorithms for Smooth Convex-Concave Minimax Problems with ${O}(1/k^2)$ Rate on Squared Gradient Norm}, booktitle={International Conference on Machine Learning}, author={Yoon, Taeho and Ryu, Ernest K.}, year={2021}, pages={12098–12109}, language={en} }

@inproceedings{Azizian_Scieur_Mitliagkas_Lacoste-Julien_Gidel_2020, title={Accelerating Smooth Games by Manipulating Spectral Shapes}, booktitle={International Conference on Artificial Intelligence and Statistics}, author={Azizian, Waïss and Scieur, Damien and Mitliagkas, Ioannis and Lacoste-Julien, Simon and Gidel, Gauthier}, year={2020}, language={en} }

@book{horn2012matrix,
  title={Matrix analysis},
  author={Horn, Roger A and Johnson, Charles R},
  year={2012},
  publisher={Cambridge University Press}
}

@article{nemirovsky1991optimality,
  title={On optimality of {K}rylov's information when solving linear operator equations},
  author={Nemirovsky, Arkadi S},
  journal={Journal of Complexity},
  volume={7},
  number={2},
  pages={121--130},
  year={1991},
  publisher={Academic Press}
}

@article{nemirovsky1992information,
  title={Information-based complexity of linear operator equations},
  author={Nemirovsky, Arkadi S},
  journal={Journal of Complexity},
  volume={8},
  number={2},
  pages={153--175},
  year={1992},
  publisher={Academic Press, Inc. Orlando, FL, USA}
}

@inbook{Rockafellar_1970,  title={Monotone operators associated with saddle-functions and minimax problems}, volume={18.1}, booktitle={Proceedings of Symposia in Pure Mathematics}, publisher={American Mathematical Society}, author={Rockafellar, R. T.}, year={1970}, pages={241–250}, language={en} }

@inproceedings{ibrahim2020linear,
  title={Linear lower bounds and conditioning of differentiable games},
  author={Ibrahim, Adam and Azizian, Wa{\i}ss and Gidel, Gauthier and Mitliagkas, Ioannis},
  booktitle={International Conference on Machine Learning},
  pages={4583--4593},
  year={2020},
}

@inproceedings{Lin_Jin_Jordan_2020, title={On Gradient Descent Ascent for Nonconvex-Concave Minimax Problems}, booktitle={International Conference on Machine Learning},  author={Lin, Tianyi and Jin, Chi and Jordan, Michael}, year={2020}, pages={6083–6093}, language={en} }

@inproceedings{Lu_Singh_Chen_Chen_Hong_2019,
  title={Alternating gradient descent ascent for nonconvex min-max problems in robust learning and {GANs}},
  author={Lu, Songtao and Singh, Rahul and Chen, Xiangyi and Chen, Yongxin and Hong, Mingyi},
  booktitle={Asilomar Conference on Signals, Systems, and Computers},
  year={2019},
}

@inproceedings{Lee_Cho_Yun_2024,
  title={Fundamental benefit of alternating updates in minimax optimization},
  author={Lee, Jaewook and Cho, Hanseul and Yun, Chulhee},
  booktitle={International Conference on Machine Learning},
  pages={26439-26514},
  year={2024},
}

@inproceedings{Zhang_Wang_Lessard_Grosse_2022, title={Near-optimal Local Convergence of Alternating Gradient Descent-Ascent for Minimax Optimization}, booktitle={International Conference on Artificial Intelligence and Statistics}, author={Zhang, Guodong and Wang, Yuanhao and Lessard, Laurent and Grosse, Roger B.}, year={2022} }

@inproceedings{Chae_Kim_Kim_2023, title={Two-timescale Extragradient for Finding Local Minimax Points},
  booktitle={International Conference on Learning Representations},author={Chae, Jiseok and Kim, Kyuwon and Kim, Donghwan}, year={2024} }

@book{rockafellar1970convex,
  title={Convex Analysis},
  author={Rockafellar, Ralph Tyrell},
  year={1970},
  publisher={Princeton University Press}
}

@book{ryu2022large,
  title={Large-scale convex optimization: algorithms \& analyses via monotone operators},
  author={Ryu, Ernest K and Yin, Wotao},
  year={2022},
  publisher={Cambridge University Press}
}

@book{conway2012functions,
  title={Functions of one complex variable II},
  author={Conway, John B},
  volume={159},
  year={2012},
  publisher={Springer Science \& Business Media}
}

@book{nesterov-survey,
	title={Introductory lectures on convex optimization: A basic course},
	author={Nesterov, Yurii},
	volume={87},
	year={1998},
	publisher={Springer Science \& Business Media}
}

@book{ransford1995potential,
  title={Potential theory in the complex plane},
  author={Ransford, Thomas},
  year={1995},
  publisher={Cambridge University Press}
}

@book{trefethen2019approximation,
  title={Approximation theory and approximation practice, extended edition},
  author={Trefethen, Lloyd N},
  year={2019},
  publisher={SIAM}
}

@book{stein2010complex,
  title={Complex analysis},
  author={Stein, Elias M and Shakarchi, Rami},
  volume={2},
  year={2010},
  publisher={Princeton University Press}
}

@article{saff2010logarithmic,
  title={Logarithmic potential theory with applications to approximation theory},
  author={Saff, Edward B},
  journal={Preprint at arXiv:1010.3760},
  year={2010}
}

@book{rivlin2020chebyshev,
  title={{C}hebyshev polynomials},
  author={Rivlin, Theodore J},
  year={2020},
  publisher={Courier Dover Publications}
}

@book{mason2002chebyshev,
  title={{C}hebyshev polynomials},
  author={Mason, John C and Handscomb, David C},
  year={2002},
  publisher={Chapman and Hall/CRC}
}

@article{walsh1926grad,
  title={{\"U}ber den Grad der Approximation einer analytischen Funktion},
  author={Walsh, Joseph L},
  year={1926},
  publisher={Verlagd. Bayer. Akad. d. Wiss.}
}

@book{bernstein1912ordre,
  title={Sur l'ordre de la meilleure approximation des fonctions continues par des polyn{\^o}mes de degr{\'e} donn{\'e}},
  author={Bernstein, Serge},
  volume={4},
  year={1912},
  publisher={Hayez, imprimeur des acad{\'e}mies royales}
}

@article{wang2020improved,
  title={Improved algorithms for convex-concave minimax optimization},
  author={Wang, Yuanhao and Li, Jian},
  journal={Advances in Neural Information Processing Systems},
  volume={33},
  pages={4800--4810},
  year={2020}
}

@article{kovalev2022first,
  title={The first optimal algorithm for smooth and strongly-convex-strongly-concave minimax optimization},
  author={Kovalev, Dmitry and Gasnikov, Alexander},
  journal={Advances in Neural Information Processing Systems},
  volume={35},
  pages={14691--14703},
  year={2022}
}

@article{heusel2017gans,
  title={{GAN}s trained by a two time-scale update rule converge to a local nash equilibrium},
  author={Heusel, Martin and Ramsauer, Hubert and Unterthiner, Thomas and Nessler, Bernhard and Hochreiter, Sepp},
  journal={Advances in Neural Information Processing Systems},
  volume={30},
  year={2017}
}
	
\end{document}